\documentclass[12pt]{amsart}
\usepackage{amsmath}
\usepackage{amssymb}
\usepackage{bbm}
\usepackage{pdfsync}
\usepackage{hyperref}
\usepackage{cleveref}
\usepackage{mathtools}
\usepackage{xcolor}
\usepackage{xparse}
\usepackage{mathrsfs}
\usepackage{mathabx}

\date{\today}
\allowdisplaybreaks
\newcommand{\R}{{\mathbb {R}}}

\newcommand{\N}{{\mathbb N}}
\newcommand{\Z}{{\mathbb Z}}

\newcommand{\M}{\mathcal M}

\newcommand{\cosec}{{\mathrm{cosec}\;}}
\newcommand{\vertiii}[1]{{\left\vert\kern-0.25ex\left\vert\kern-0.25ex\left\vert #1 
		\right\vert\kern-0.25ex\right\vert\kern-0.25ex\right\vert}}

\DeclarePairedDelimiterX{\set}[1]{\{}{\}}{\setargs{#1}}
\NewDocumentCommand{\setargs}{>{\SplitArgument{1}{;}}m}
{\setargsaux#1}
\NewDocumentCommand{\setargsaux}{mm}
{\IfNoValueTF{#2}{#1} {#1\,\delimsize|\,\mathopen{}#2}}%{#1\:;\:#2}

\newtheorem{theorem}{Theorem}[section]
\newtheorem{lemma}[theorem]{Lemma}
\newtheorem{remark}[theorem]{Remark}

\newtheorem{coro}[theorem]{Corollary}
\newtheorem{proposition}[theorem]{Proposition}

\newtheorem*{theorem*}{Theorem}

\theoremstyle{definition}

\newtheorem{example}[theorem]{Example}

\numberwithin{equation}{section}

\begin{document}
\title[Bilinear Bochner-Riesz means and Kakeya Maximal function]{Bilinear Bochner-Riesz means for convex domains and Kakeya Maximal function}

\author{Ankit Bhojak}
\address{Ankit Bhojak\\
	Department of Mathematics\\
	Indian Institute of Science Education and Research Bhopal\\
	Bhopal-462066, India.}
\email{ankitb@iiserb.ac.in}

\author{Surjeet Singh Choudhary}
\address{Surjeet Singh Choudhary\\
	Department of Mathematics\\
	Indian Institute of Science Education and Research Bhopal\\
	Bhopal-462066, India.}
\email{surjeet19@iiserb.ac.in}

\author{Saurabh Shrivastava}
\address{Saurabh Shrivastava\\
	Department of Mathematics\\
	Indian Institute of Science Education and Research Bhopal\\
	Bhopal-462066, India.}
\email{saurabhk@iiserb.ac.in}

\thanks{}
\begin{abstract} In this paper we introduce bilinear Bochner-Riesz means associated with convex domains in the plane $\R^2$ and study their $L^p-$boundedness properties for a wide range of exponents. One of the important aspects of our proof involves the use of bilinear Kakeya maximal function in the context of bilinear Bochner-Riesz problem. This amounts to establish suitable $L^p-$ estimates for the later. We also point out some natural connections between bilinear Kakeya maximal function and Lacey's bilinear maximal function. 
	
\end{abstract}
\subjclass[2010]{Primary 42B15, 42B25}	
\maketitle
\tableofcontents
\section{Introduction}
The study of bilinear Bochner-Riesz means has become an active area of research in harmonic analysis in recent years. The bilinear Bochner-Riesz mean of index $\lambda\geq 0$ is the bilinear multiplier operator defined by 
\begin{eqnarray*}
	\mathcal{B}^{\lambda}(f,g)(x)&:=&\int_{\mathbb{R}^{2n}}(1-|\xi|^{2}-|\eta|^{2})^{\lambda}_{+}\hat{f}(\xi)\hat{g}(\eta)e^{2\pi i x\cdot(\xi+\eta)}~~d\xi d\eta,
\end{eqnarray*}
where $f, g\in \mathcal S(\R^n)-$~the Schwartz class on $\R^n$, and  $(1-|\xi|^{2}-|\eta|^{2})^{\lambda}_{+}=(1-|\xi|^{2}-|\eta|^{2})^{\lambda}\chi_{_D}(\xi,\eta)$. The notation $D$ stands for the unit ball in $\R^n\times \R^n$. 

We are interested in the study of $L^{p_{1}}(\mathbb{R}^{n})\times L^{p_{2}}(\mathbb{R}^{n})\rightarrow L^{p_3}(\mathbb{R}^{n})$ boundedness of the operator $\mathcal{B}^{\lambda}$, i.e., estimates of the form 
\begin{eqnarray}\label{Lp}
\|\mathcal{B}^{\lambda}(f,g)\|_{p_3}\lesssim \|f\|_{p_1} \|g\|_{p_2},
\end{eqnarray} 
for all $f$ and $g$ in a suitable class of functions in $L^{p_1}(\R^n)\times L^{p_2}(\R^n)$ with the implicit constant independent of $f$ and $g$. 
Moreover, we shall always assume that the exponents $p_1,p_2$ and $p_3$ in \eqref{Lp} satisfy the H\"{o}lder relation $\frac{1}{p_1}+\frac{1}{p_2}=\frac{1}{p_3}$ with  $1\leq p_1, p_2\leq \infty$.

First results for the operator $\mathcal{B}^{\lambda}$ were obtained by Bernicot and Germain~\cite{BG} for $n=1$. These were later improved upon and extended to higher dimensions by~Bernicot, Grafakos, Song and Yan~\cite{Bernicot}. In the case of dimension $n=1,$ they gave a complete picture of $L^p-$boundedness of $\mathcal{B}^{\lambda}, \lambda>0,$ for exponents in the Banach triangle $\{(p_1,p_2,p_3):1\leq p_1,p_2,p_3\leq \infty\}$. However, in higher dimensions the results were not sharp. Jeong and Lee~\cite{JSV} established sharp results for a certain range of exponents $p_1,p_2,p_3$ and index $\lambda>0$. In particular, they showed that $\mathcal{B}^{\lambda}$ maps $L^2(\R^n)\times L^2(\R^n)$ into $L^1(\R^n)$ for all $\lambda>0$, which is sharp in $\lambda$. Recently, Kaur and Shrivastava~\cite{KS} obtained new results for the operator~$\mathcal{B}^{\lambda}$ and its maximal variant. The results in~\cite{KS} are the best known so far and also are sharp in some cases. Since the results are technical and require new notation, which are not needed otherwise in the rest of the paper, we skip the details here. We would like to refer to Liu and Wang~\cite{LW} for results in non-Banach triangle, i.e. when $p_3<1$, to Choudhary, Kaur, Shrivastava and Shuin~\cite{CKSS} and Choudhary and Shrivastava~\cite{CS} for results about bilinear Bochner-Riesz square function and its applications to bilinear multipliers. 

The case of $\lambda=0$ is more subtle. This was first addressed by Grafakos and Li~\cite{GL} in dimension $n=1$. They proved that the operator $\mathcal{B}^{0}$, commonly referred to as the bilinear disc multiplier operator, maps  $L^{p_{1}}(\mathbb{R})\times L^{p_{2}}(\mathbb{R})$ into $L^{p_3}(\mathbb{R})$ for exponents in local $L^2-$ range: $\{(p_1,p_2,p_3):2\leq p_1,p_2,p_3'<\infty\}$. In general, the question whether characteristic function of a given geometric shape in $\R^2$ gives rise to bilinear multiplier has been addressed for many interesting shapes. In the seminal papers Lacey and Thiele~\cite{LT1,LT2} proved that the bilinear Hilbert transform associated with the characteristic function $\chi_{_{H_{\alpha}}}(\xi,\eta)$ of half plane $H_{\alpha}=\{(\xi,\eta)\in \R^2:\xi-\alpha\eta>0\}, \alpha \in \R,$ satisfies ~\eqref{Lp} for a wide range of exponents $p_1,p_2$ and $p_3$.  
Demeter and Gautam~\cite{DG} proved estimate~\eqref{Lp} for the bilinear operator associated with infinite lacunary polygon inscribed in the disc. Through infinite lacunary polygon they tried to  approximate the boundary of disc at a point and showed that positive results can be obtained for such bilinear multipliers even outside the local $L^2-$range. Next, we refer to~Muscalu~\cite{Muscalu}, where ~estimate \eqref{Lp} is proved for bilinear multipliers determined by graph of convex functions with bounded slopes. Recently, Saari and Thiele~\cite{ST} studied bilinear paraproducts associated with certain convex sets. In particular, they showed that the bilinear operator associated with multiplier symbol $\chi_{_C}(\xi,\eta),$ where $C$ is the convex set  $\{(\xi,\eta)\in \R^2: \xi\leq 0~~\text{and}~~2^{\xi}\leq \eta<1\},$ satisfies estimate~\eqref{Lp} for exponents  $p_1,p_2$ and $p_3$ in the local $L^2-$range.

Motivated by these recent developments in the direction of bilinear multipliers and results for Bochner-Riesz means associated with convex domains by~Seeger and Ziesler~\cite{SZ} and Cladek~\cite{Cla1, Cla2}, we plan to investigate analogous questions for bilinear Bochner-Riesz means associated with convex domains in $\R^2$. We will see that our investigation naturally gives rise to bilinear analogue of Kakeya maximal function. In this paper our aim is to  
\begin{itemize}
	\item Generalize the notion of bilinear Bochner-Riesz means in the context of open and bounded convex domains in the plane $\R^2$. The existing techniques employed to deal with bilinear Bochner-Riesz means in~\cite{Bernicot,BG,JSV,JL,KS} do not extend to the case of general convex domains as they rely on the explicit form of the multiplier $(1-|\xi|^{2}-|\eta|^{2})^{\lambda}_{+}$. The framework developed in~\cite{SZ} allows us to begin with the study of this case. This approach naturally leads to bilinear analogues of Kakeya maximal functions. In~\cite{BG}, this connection was pointed out briefly for the bilinear multiplier $(1-|\xi|^{2}-|\eta|^{2})^{\lambda}_{+}.$ We develop this approach systematically in this paper. 
	\item Extend the classical approach of using geometric maximal functions to prove $L^p-$boundedness results for Fourier multipliers to the bilinear setting. In this direction we introduce the bilinear Kakeya maximal function in the plane and study its $L^p-$boundedness properties. We establish a connection between bilinear Bochner-Riesz means and Kakeya maximal function and consequently deduce the $L^p-$ estimates for the multiplier operator. See Section~\ref{sec:bkmf} for definitions and results. 
	\end{itemize}
\section{Main results}
\subsection{Bilinear Bochner-Riesz means associated with convex domains}\label{sec:bbr}
Let $\Omega$ be an open and bounded convex set in the plane $\R^2$ containing the origin. Let $\partial\Omega$ denote the boundary of $\Omega$. Consider the Minkowski functional associated with $\Omega$ given by,
\[\rho(\xi,\eta)=\inf\{t>0:\;(t^{-1}\xi,t^{-1}\eta)\in\partial\Omega\}.\]
The bilinear Bochner-Riesz mean of index $\lambda>0$ associated with the convex domain $\Omega$ is defined by
\[\mathcal{B}_{\Omega}^\lambda(f,g)(x)=\int_\R\int_\R (1-\rho(\xi,\eta))^\lambda_+\hat f(\xi)\hat g(\eta)e^{2\pi ix(\xi+\eta)}\;d\xi d\eta.\]
Observe that if $\Omega$ is the unit disc in $\R^2$, the operator $\mathcal{B}_{\Omega}^\lambda$ is same as $\mathcal{B}^\lambda$ defined as earlier. 

The following theorem is the main result of this paper for bilinear Bochner-Riesz means $\mathcal{B}_{\Omega}^\lambda.$
\begin{theorem}\label{BR}
Let $\Omega$ be an open and bounded convex set and $\mathcal{B}_{\Omega}^\lambda$ be the bilinear Bochner-Riesz mean described as above. Then for $\lambda>0$ and exponents $p_1,p_2,p_3$ satisfying $p_1,p_2\geq 2$ and $\frac{1}{p_1}+\frac{1}{p_2}=\frac{1}{p_3}$, the operator $\mathcal{B}_{\Omega}^\lambda$ maps $L^{p_1}(\R)\times L^{p_2}(\R)$ into $L^{p_3}(\R)$, i.e., there exists a constant $C=C(\Omega,\lambda, p_1, p_2)>0$ such that for all $f, g \in \mathcal S(\R)$ we have 
\begin{equation*}
	\|\mathcal{B}_{\Omega}^\lambda(f,g)\|_{p_3}\leq C\|f\|_{p_1}\|g\|_{p_2}.
\end{equation*} 
\end{theorem}
As an application of \Cref{BR}, we obtain the following result for  quasiradial bilinear multipliers. 
\begin{coro}\label{cor:qrm}
	Let $\rho$ be the Minkowski functional associated with a convex domain $\Omega$ described as above and $m:[0,\infty)\to\R$ be a function such that for $\lambda>0$, the following holds 
	\[\int_0^\infty s^\lambda |m^{\lambda+1}(s)|ds<\infty.\]
	Then the bilinear multiplier $T_m$ defined as
	\[T_m(f,g)(x)=\int_\R\int_\R m(\rho(\xi,\eta))\hat f(\xi)\hat g(\eta)e^{2\pi ix(\xi+\eta)}\;d\xi d\eta\]
maps $L^{p_1}(\R)\times L^{p_2}(\R)$ into $L^{p_3}(\R)$ for all $p_1,p_2\geq 2$ and $\frac{1}{p_1}+\frac{1}{p_2}=\frac{1}{p_3}$.
\end{coro}
The \Cref{cor:qrm} is a direct consequence of \Cref{BR} along with the well-known subordination formula given below, see  \cite{Tre} for more details. 
\[(m\circ\rho)(\xi,\eta)=\frac{(-1)^{\lfloor \lambda\rfloor+1}}{\Gamma(\lambda+1)}\int_0^\infty s^\lambda m^{\lambda+1}(s)\left(1-\frac{\rho(\xi,\eta)}{s}\right)^\lambda_+ds.\]
\subsection{Bilinear Kakeya maximal function}\label{sec:bkmf}
In this section we describe the main results of the paper for bilinear Kakeya maximal function.  

Let $\mathfrak{F}$ be a collection of finite measure sets in $\R^n$. Consider the maximal averaging operator associated with the collection $\mathfrak{F}$ defined by
\begin{equation}\label{linearF}
	M_\mathfrak{F}f(x)=\sup\limits_{F\in\mathfrak{F}:\;x\in F}\frac{1}{|F|}\int_F|f(y)|\;dy.
\end{equation}
Maximal averaging operators play key roles in differentiation theory. Under certain geometric conditions on the sets in $\mathfrak{F}$, the operator $M_\mathfrak{F}$ enjoys $L^p-$boundedness properties. For example, if $\mathfrak{F}$ is the collection of cubes (or balls) in $\R^n$, the operator $M_\mathfrak{F},$ commonly known as the Hardy-Littlewood maximal operator, maps $L^p(\R^n)$ into itself for all $1<p\leq\infty$ with a weak-type boundedness at $p=1$. However, if $\mathfrak{F}$  is the collection of all rectangles in $\R^n$, then by a well-known Besicovitch set construction, see~\cite{book:stein}, it is known that the corresponding operator $M_\mathfrak{F}$ fails to be $L^p-$bounded for all $1\leq p<\infty.$  

The Kakeya maximal function involves the averages over rectangle with an extra condition on sides of rectangle. In this paper we will restrict ourselves to Kakeya maximal function in dimension $n=2$. 

For an integer $N>1$ and $\delta>0$, let $\mathcal R_{\delta,N}$ be the class of all rectangles in $\R^2$ with dimensions $\delta\times\delta N$ and $\mathcal R_N=\cup_{\delta>0}\mathcal R_{\delta,N}$. A standard dilation argument implies that the $L^p-$boundedness of a fixed scale maximal operator $M_{\mathcal R_{\delta,N}}$ is equivalent to that of the operator $M_{\mathcal R_{1,N}}$. C\'ordoba \cite{Co} proved that 
\begin{equation}\label{Cordoba}
	\|M_{\mathcal R_{1,N}}\|_{L^2\to L^2}\lesssim (\log N)^\frac{1}{2},
\end{equation}
and the logarithmic dependence on the ``eccentricity" $N$ is sharp.

Later, Str\"omberg \cite{Str1} proved the following sharp bounds for the maximal operator $M_{\mathcal R_N}$. 
\begin{equation}\label{Stromberg}
	\|M_{\mathcal R_N}\|_{L^2\to L^2}\lesssim \log N
%	\|M_{B_N}\|_{L^2(\R^2)\to L^{2,\infty}(\R^2)}&\lesssim (\log N)^\frac{1}{2},
\end{equation}
We consider the bilinear analogue of Kakeya maximal functions defined above. The fixed scale bilinear Kakeya maximal function is defined by 
$$\mathcal{M}_{\mathcal R_{\delta,N}}(f,g)(x)=\sup_{\substack{R\in \mathcal R_{\delta,N}\\(x,x)\in R}}\frac{1}{|R|}\int_R |f(y_1)||g(y_2)|dy_1 dy_2.$$
The bilinear Kakeya maximal function associated with the collection $\mathcal R_N$ is defined by
\[\mathcal{M}_{\mathcal R_N}(f,g)(x)=\sup_{k\leq N}\sup_{\substack{R\in \mathcal R_k\\(x,x)\in R}}\frac{1}{|R|}\int_R |f(y_1)||g(y_2)|dy_1dy_2,\]
The bilinear Kakeya maximal functions arise naturally in the study of the bilinear Bochner-Riesz means. Therefore, sharp $L^p-$estimates for the maximal functions yield the corrsponding $L^p-$estimates for Bochner-Riesz means. 
\begin{remark}\label{rem:tanaka} Formally, the bilinear Kakeya maximal function $\mathcal{M}_{\mathcal R_{1,N}}(f,g)$ may also be obtained by restricting the (linear) two-dimensional Kakeya maximal function $M_{\mathcal R_{1,N}}(f \otimes g)$ to the diagonal $\{(x,x):x\in \R\}$, where $(f \otimes g)(x,y)=f(x)g(y).$ 
\end{remark}
We have the following result for the operator $\mathcal{M}_{\mathcal R_{\delta,N}}$. As earlier, it is enough to consider the case of $\delta=1$. 
\begin{theorem}\label{KN}
	Let $1\leq p_1,p_2\leq\infty$ be such that $\frac{1}{p_3}=\frac{1}{p_1}+\frac{1}{p_2}$, then we have the following estimates. 
	\begin{enumerate}
		\item {\bf Banach case:} \begin{enumerate}
			\item For $1\leq p_1,p_2,p_3\leq\infty$,  $\mathcal{M}_{\mathcal R_{1,N}}$ maps $L^{p_1}(\R)\times L^{p_2}(\R)$ to $L^{p_3,\infty}(\R)$ with operator norm bounded by a constant independent of $N$. Note that standard bilinear interpolation arguments yield strong type bounds for all $p_3>1$ with operator norm independent of $N$.  
			\item $\mathcal{M}_{\mathcal R_{1,N}}$ maps $L^{p}(\R)\times L^{p'}(\R)\to L^{1}(\R)$ with operator norm bounded by a constant multiple of $(\log N)^\frac{1}{\min\{p_1,p_2\}}$. Here $p'$ denotes the conjugate index $\frac{1}{p}+\frac{1}{p'}=1.$
		\end{enumerate}  
		\item {\bf Non-Banach case:} 
		\begin{enumerate} \item For $1\leq p_1,p_2\leq\infty$ and $\frac{1}{2}<p_3\leq1$, $\mathcal{M}_{\mathcal{R}_{1,N}}$ is bounded from $L^{p_1}(\R)\times L^{p_2}(\R)$ to $L^{p_3}(\R)$ with constant $N^{\frac{1}{p_3}-1}$.
		\item {\bf End-point $(1,1,1/2)$:} $\mathcal{M}_{\mathcal{R}_{1,N}}$ maps $L^1(\R)\times L^1(\R)$ to $L^{\frac{1}{2},\infty}(\R)$ with constant $N$.
		\end{enumerate}
			%\item $\mathcal{M}_{\mathcal{R}_{1,N}}$ maps $L^1(\R)\times L^1(\R)$ to $L^{\frac{1}{2}}(\R)$ with constant $N\log N$.
		\end{enumerate}
\end{theorem}
Next, we have the following $L^p-$boundedness result for the operator $\mathcal{M}_{\mathcal R_N}$.
\begin{theorem}\label{MN}Let $1\leq p_1,p_2\leq\infty$ and $\frac{1}{p_3}=\frac{1}{p_1}+\frac{1}{p_2}$. The following bounds hold true.
	\begin{enumerate}
		\item {\bf Banach case:} 
		\begin{enumerate}
			\item For all $1\leq p_1,p_2,p_3 \leq \infty$ we have that $\|\mathcal{M}_{\mathcal R_N}\|_{L^{p_1}\times L^{p_2} \rightarrow L^{p_3,\infty}}\lesssim 1.$ Observe that as a consequence of this we get strong type bounds 
				$\|\mathcal{M}_{\mathcal R_N}\|_{L^{p_1}\times L^{p_2} \rightarrow L^{p_3}}\lesssim 1$ for all $p_3>1$.
			\item For all $1<p_1,p_2\leq\infty$ we have that $\|\mathcal{M}_{\mathcal R_N}\|_{L^{p_1}\times L^{p_2} \rightarrow L^1 }\lesssim \log N.$ Moreover, the bound $\log N$ is sharp here. 
		\end{enumerate}
		\item {\bf Non-Banach case:} 
		\begin{enumerate}
			\item For $1< p_1,p_2\leq\infty$ and $\frac{1}{2}<p_3<1$, we have $\|\mathcal{M}_{\mathcal R_N}\|_{L^{p_1}\times L^{p_2} \rightarrow L^{p_3}}\lesssim N^{\frac{1}{p_3}-1}$.
			\item {\bf End-point case:} If atleast one of $p_1$ or $p_2$ is $1$ then  $\|\mathcal{M}_{\mathcal R_N}\|_{L^{p_1}\times L^{p_2} \rightarrow L^{p_3,\infty}}\lesssim N$. 
		\end{enumerate}
		\end{enumerate}
\end{theorem}
In \Cref{sec:examples} we will provide some examples towards the sharpness of constants with respect to the parameter $N$ in the theorems above. 
\begin{remark} Indeed, in the proof of \Cref{BR} we will require to consider bilinear Kakeya maximal function over rectangles whose eccentricity is less than or equal to $N$. We will denote such a maximal function by  $\mathcal{M}_{\mathcal R_{\leq N}}$. An analogue of \Cref{MN} holds for $\mathcal{M}_{\mathcal R_{\leq N}}$ with an additional constant of $\log N$ on the operator norm of $\mathcal{M}_{\mathcal R_{\leq N}}$. Due to notational inconvenience and repetition we will skip the details. 
\end{remark}
Finally, we describe a vector-valued result for the operator $\mathcal{M}_{\mathcal R_N}$. This will be required in the proof of \Cref{BR}.
\begin{theorem}\label{vector}
	Let $1<p_1,p_2<\infty$, $1\leq p_3<\infty$ and $1<r_1,r_2\leq\infty$, $1\leq r_3\leq\infty$ satisfy $\frac{1}{p_3}=\frac{1}{p_1}+\frac{1}{p_2}$ and $\frac{1}{r_3}=\frac{1}{r_1}+\frac{1}{r_2}$. Then for any $\epsilon>0$, we have
	\begin{equation*}
		\left\|\left(\sum_{j}\left|\M_{\mathcal R_N}\left(f_{j}, g_{j}\right)\right|^{r_3}\right)^{\frac{1}{r_3}}\right\|_{p_3}\lesssim N^\epsilon\left\|\left(\sum_{j}\left|f_{j}\right|^{r_{1}}\right)^{\frac{1}{r_{1}}}\right\|_{p_{1}}\left\|\left(\sum_{j}\left|g_{j}\right|^{r_{2}}\right)^{\frac{1}{r_{2}}}\right\|_{p_{2}}.
	\end{equation*}
\end{theorem}
The proof of $L^p(R)\times L^{p'}(\R)\rightarrow L^1(\R)-$boundedness of $\M_{\mathcal R_N}$ in \Cref{MN} is based on an interpolation trick between suitably chosen exponents. The idea is motivated by the linear counterpart from~C\'ordoba~\cite{Co} and Str\"omberg~\cite{Str1}. This approach yields sharp constants with respect to $N$ in \Cref{MN}. However, for the fixed scale maximal function $\mathcal M_{\mathcal R_{\delta,N}}$ in \Cref{KN}, this method does not give sharp constant in $N$. Indeed, the connection of $\M_{\mathcal R_N}$ with linear Kakeya maximal function on product type functions as mentioned before, allows us to exploit the ideas of Tanaka~\cite{Tan} to prove better bounds in \Cref{KN}. This method is more direct and involves a certain counting argument for rectangles under consideration.

%%%%%%%%%%%%%%%
%%%%%%%%%%%%%%%
\section{Basic framework for \Cref{BR}}\label{sec:proofofBR}
In this section we develop the basic framework required to deal with general convex domains. This part involves new definitions, reduction of the problem to smooth domains and parametrization of the boundary~$\partial \Omega.$ We mostly follow Seeger and Ziesler~\cite{SZ} for this part.  
\subsection{Reduction to domains with smooth boundary}\label{sec:redtosmooth} 
First, observe that using a standard dilation argument for bilinear multipliers, we may without loss of generality, assume that 
\begin{equation*}
	B(0,4)\subset \Omega\subset\overline{\Omega}\subset B(0,2^M),
\end{equation*} where $M\geq 3$ is a fixed constant.

Next, observe that the boundary $\partial\Omega$ may not be smooth. At this point we invoke the approach used by Seeger and Ziesler~\cite{SZ} in the linear case. This allows us to reduce the problem to domains with smooth boundary. We approximate $\partial\Omega$ by a sequence of smooth curves using polygons whose boundary is smoothened near the vertices. We require some preliminary definitions in the context of convex domains in order to perform this reduction. 

Given a point $P\in\partial\Omega$, we say that a line $\ell$ passing through $P$ is a supporting line for $\Omega$ at $P$ if $\Omega$ is contained in the closed half plane whose boundary is the line $\ell$. Let  $\textit{T}(\Omega,P)$ denote the set of all supporting lines for $\Omega$ at $P$. Note that if $\partial\Omega$ is $C^1-$smooth, the tangent at $P$ is the unique supporting line for $\Omega$ at $P$.

For $\delta>0$, consider the ball centered at $P$ on $\partial\Omega$ along $\ell$ given by  $$B(P,\ell,\delta)=\{X\in\partial\Omega:\;\mathrm{dist}(X,\ell)<\delta\}.$$ Denote  the collection of such balls by $\mathfrak{N}_\delta=\{B(P,\ell,\delta):\;P\in\partial\Omega,\ell\in\textit{T}(\Omega,P)\}$. Let $\textit{N}(\Omega,\delta)$ be the minimum number of balls in $\mathfrak{N}_\delta$ required to cover the boundary $\partial\Omega$. The upper Minkowski dimension $\kappa_\Omega$ of $\Omega$ is defined by 
\begin{equation}\label{kappa}
	\kappa_\Omega=\limsup_{\delta\to 0}\frac{\log \textit{N}(\Omega,\delta)}{\log \delta^{-1}}.
\end{equation}
Note that for any convex set $\Omega$, we have $0\leq\kappa_\Omega\leq\frac{1}{2}$. Further, $\kappa_\Omega=0$ if $\Omega$ is a convex polygon and $\kappa_\Omega=\frac{1}{2}$ if $\Omega$ is a smooth domain, for example, if $\Omega$ is the unit ball, then $\kappa_\Omega=\frac{1}{2}$. 

With these notions we are ready to invoke the approximation lemma  from~\cite{SZ}. 
\begin{lemma}{\cite{SZ}}\label{smoothOmega}
	There exists a sequence of domains $\Omega_n$ whose boundary $\partial\Omega_n$ is $C^\infty-$smooth and the Minkowski functional $\rho_n$ corresponding to $\Omega_n$ satisfy the following conditions. 
	\begin{enumerate}
		\item $\Omega_n\subseteq\Omega_{n+1}\subset\Omega$ and $\Omega=\bigcup_n\Omega_n$.
		\item $\rho(\xi)\leq\rho_{n+1}(\xi)\leq\rho_n(\xi)$ with $\rho_n(\xi)-\rho(\xi)\leq2^{-n-1}\rho(\xi)$. In particular $\lim\limits_{n\to\infty}\rho_n(\xi)=\rho(\xi)$ with uniform convergence on compact sets.
		\item If $\delta\geq2^{-n+2}$, then
		\[N(\Omega_n,2\delta)\lesssim N(\Omega,\delta).\]
	\end{enumerate}
\end{lemma}
Observe that it is enough to prove \Cref{BR} for domain $\Omega_n$ as in the lemma above with bounds uniform in $n$. Then, \Cref{BR} for the domain $\Omega$ follows using Fatou's lemma.  
\subsection{Decomposition of the boundary $\partial\Omega$}\label{refinement}
Following the approach of Seeger and Ziesler~\cite{SZ} we consider the following  parametrization of the smooth boundary $\partial\Omega$. 
\begin{lemma}{\cite{SZ}}\label{para}
	Let $\{u_p\}_{p=1}^{2^{2M}}$ be the set of $2^{2M}$ uniformly distributed unit vectors in $\R^2$ and $\mathfrak{G}_{u_p}=\{(\xi,\eta)\in\R^2:\langle (\xi,\eta),u_p\rangle\leq 0,\;|\langle (\xi,\eta),u_p^\perp\rangle|\leq 2\}$ be the half strip associated with $u_p$. We can parametrize  $\partial\Omega\cap\mathfrak{G}_{u_p}$ by
	\[t\mapsto tu_p^\perp+\gamma(t)u_p,\;-2\leq t\leq 2,\]
	where $\gamma:[-2,2]\to [-2^M,-2]$ is a convex function with left and right derivatives $\gamma_L'$ and $\gamma_R'$ satisfying,
	\[-2^{M-1}\leq\gamma_L'(t)\leq\gamma_R'(t)\leq 2^{M-1},\;-2\leq t\leq 2.\]
	Moreover, for a supporting line $\ell$ at the point $P\in\partial\Omega$ and a outward unit normal vector $\vec n$ we have 
	\begin{equation}\label{tangent}
		\left\langle \frac{P}{|P|},\vec n\right\rangle\geq 2^{-M}.
	\end{equation}
\end{lemma}
Next, for a given $\delta>0$, we decompose the boundary $\partial\Omega\cap\mathfrak{G}_{u_p}$ into pieces such that the kernel corresponding to each piece is integrable and its growth is controlled by the covering number $N(\Omega,\delta)$. As in~\cite{SZ} consider a partition $\mathfrak{U}_{u_p}(\delta)=\{-1=a_0<a_1<\dots<a_{Q_{u_p}}=1\}$ of $[-1,1]$ such that for $j=0,\dots,Q_{u_p}(\delta)-1$, we have
\[(a_{j+1}-a_j)(\gamma_L'(a_{j+1})-\gamma_R'(a_j))\leq\delta,\]
and
\[(t-a_j)(\gamma_L'(t)-\gamma_R'(a_j))\leq\delta,\text{ if }t>a_{j+1}.\]
The following lemma gives a control on quantity $Q_{u_p}(\delta)$ with respect to the covering number $N(\Omega,\delta)$. 
\begin{lemma}{\cite{SZ}}
	There exists a constant $C_M>0$ such that
	\begin{enumerate}
		\item $Q_{u_p}(\delta)\leq C_M\delta^{-\frac{1}{2}}$.
		\item $C_M^{-1}N(\Omega,\delta)\leq\sum_{p=1}^{2^{2M}}Q_{u_{p}}(\delta)\leq C_M N(\Omega,\delta)\log\delta^{-1}$.
	\end{enumerate}
\end{lemma}
We need to refine the partition further in order to obtain sharper estimates for the underlying kernels. For each fixed $j$, consider points $\{a_{j,\nu}:\;\nu=-2M-l,\dots,2M+l\}$ such that
\begin{itemize}
	\item For any interval $A_{j,\nu}=[a_{j,\nu},a_{j,\nu+1}]$, we have $|A_{j,\nu}|\geq 2^{-5M}\delta$.
	\item For any two consecutive intervals $A_{j,\nu}$ and $A_{j',\nu'}$, we have,
	\begin{equation}\label{interval}
	(t-s)(\gamma_L'(t)-\gamma_R'(s))\leq\delta,\text{ if }t<s,\;\text{and}\;t,s\in A_{j,\nu}\cup A_{j',\nu'}.
	\end{equation}
\end{itemize}  
\section{Auxiliary results for \Cref{BR}}\label{sec:aux}
In this section we discuss some supporting results which will be required in proving the main result \Cref{BR}. The first lemma is well-known and is an easy consequence of Minkowski's integral inequality. It says that integrability of the kernel of a bilinear multiplier operator is sufficient for it to be bounded from $L^{p_1}(\R)\times L^{p_2}(\R)$ to $L^{p_3}(\R)$ for $p_3\geq 1$. 
\begin{lemma}\label{intker}
	Let $S_m$ be a bilinear operator associated with the multiplier $m$ defined as
	\[S_m(f,g)(x)=\int_{\R}\int_{\R} m(\xi,\eta)\hat f(\xi)\hat g(\eta)e^{2\pi ix(\xi+\eta)}d\xi d\eta.\]
	Suppose $\|\mathcal{F}^{-1} m\|_{L^1(\R^2)}<\infty$, then for $p_1,p_2,p_3\geq 1,\;\frac{1}{p_1}+\frac{1}{p_2}=\frac{1}{p_3}$, we have,
	\[\|S_m\|_{L^{p_1}\times L^{p_2}\to L^{p_3}}\lesssim\|\mathcal{F}^{-1} m\|_{1}.\]
	Here the notation $\mathcal{F}^{-1}$ stands for the inverse Fourier transform. 
\end{lemma}
Next, we recall a lemma from~\cite{SZ} which provides us with local integrability estimates for the multipliers under consideration. 
\begin{lemma}{\cite{SZ}}\label{SZlemma}
	Let $h:[0,\infty)\to\R$ be an absolutely continuous function such that $\lim\limits_{t\to\infty}h(t)=0$ and $\|th'(t)\|_{L^1[0,\infty)}<0$. Suppose the function
	\[F(\tau)=\int_0^\infty h'(s)e^{is\tau}\;ds\]
	satisfies $|F(\tau)|+|F'(\tau)|\lesssim (1+|\tau|)^{-2}$.
	Let $A_k=B(0,2^k)\setminus B(0,2^{k-1}),\;k\geq1$. Then we have,
	\begin{align*}
		\|\mathcal{F}^{-1}(h\circ\rho)\|_{L^1(B(0,1))}&\lesssim 1,\;\text{and}\\
		\|\mathcal{F}^{-1}(h\circ\rho)\|_{L^1(A_k)}&\lesssim k2^{-k}.
	\end{align*}
\end{lemma}
The following lemma provides pointwise estimates for the kernels that we encounter while proving \Cref{BR}. 
\begin{lemma}\label{kernelest}
	Let $A_k$ be the annulus as above and $l\geq 1$. If $K$ is a kernel defined in $\R^2$ such that  
	$$|K(x)|\leq \frac{a}{(1+|ax_1|)^2}\frac{1}{1+x_2^2},~ x=(x_1,x_2),$$
	for some $a\geq 2^{-l}$. Then the following holds. 
	\begin{enumerate}
		\item If $H$ is another kernel defined in $\R^2$ such that $\|H\chi_{A_k}\|_1\leq k2^{l-k}$, then 
		\[H\ast K(x)\chi_{_{\{|\cdot|> 2^{10l}\}}}(x)= \left(H\chi_{_{\{|\cdot|> 2^{5l}\}}}\right)\ast K(x)\chi_{_{\{|\cdot|> 2^{10l}\}}}(x)+L_1(x),\]
		where $\|L_1\|_1\lesssim2^{-2l}$.
		\item If $H$ satisfies $\|H\chi_{A_k}\|_1\lesssim 1$, then
		\[H\ast K(x)\chi_{\{|\cdot|\leq 2^{10l}\}}(x)= \left(H\chi_{\{|\cdot|\leq 2^{20l}\}}\right)\ast K(x)\chi_{\{|\cdot|\leq 2^{10l}\}}(x)+L_2(x)\]
		where $\|L_2\|_1\lesssim2^{-9l}$.
	\end{enumerate}
\end{lemma}
\begin{proof}
	To estimate the first inequality, we write
	\begin{align*}
			H\ast K(x)\chi_{\{|\cdot|> 2^{10l}\}}(x)&=\chi_{\{|\cdot|>2^{10l}\}}(x)\int_{\R^2}H(y)K(x-y)dy\\
			&=\chi_{\{|\cdot|>2^{10l}\}}(x)\left(\int\limits_{|y|>2^{5l}}H(y)K(x-y)dy+\int\limits_{|y|\leq 2^{5l}}K_1(y)K_2(x-y)dy\right)\\
			&=\chi_{\{|\cdot|> 2^{10l}\}}(x)\int_{\R^2}H(y)\chi_{\{|\cdot|> 2^{5l}\}}(x)K(x-y)dy\\
			& +\chi_{\{|\cdot|> 2^{10l}\}}(x)\int_{|y|\leq 2^{5l}}H(y)K(x-y)dy\\
			&= \left(H\chi_{\{|\cdot|> 2^{5l}\}}\right)\ast K(x)\chi_{\{|\cdot|> 2^{10l}\}}(x)+L_1(x).
		\end{align*}
		
		For kernel $L_1$, we have the estimate
		\begin{align*}
			\left\|L_1\right\|_1&\leq \int_{|x|>2^{10l}}\int_{|y|\leq 2^{5l}}|H(y)K(x-y)|dydx\\
			&= \int_{|y|\leq 2^{5l}}|H(y)|\int_{|x|>2^{10l}}|K(x-y)|dxdy\\
			&\leq \int_{|y|\leq 2^{5l}}|H(y)|dy\int_{|z|>2^{5l}}|K(z)|dz, 
		\end{align*}
		where we have used that $|z|=|x-y|>2^{5l}$ when $|x|>2^ {10l}$ and $|y|\leq 2^{5l}$. For the integral involving the kernel $H$, we have
		\begin{align*}
			\int_{|y|\leq 2^{5l}}|H(y)|dy&= \sum_{k=0}^{5l}\int_{A_k}|H(y)|dy\leq \sum_{k=0}^{5l}k2^{l-k}\leq l2^l.
		\end{align*}
		When $|z|>2^{5l}$, either $|z_1|>2^{5l-1}$ or $|z_2|>2^{5l-1}$. If $|z_1|>2^{5l-1}$, then 
		\begin{align*}
			\int_{|z|>2^{5l}}|K(z)|dz &\leq \int_{|z_1|>2^{5l}}\frac{a}{(1+|az_1|)^2}dz_1\int_{\R}\frac{1}{1+z_2^2}dz_2\\
			&\leq C\frac{1}{a}\int_{|z_1|>2^{5l}}\frac{1}{|z_1|^2}dz_1\\
			&\leq 2^{-4l}.
		\end{align*}
		Similarly, we can get the estimate when $|z_2|>2^{5l-1}$,
		\[\int_{|z|>2^{5l}}|K(z)|dz\leq 2^{-5l}.\]
		Therefore, we obtain that
		\[\|L_1\|_1\lesssim2^{-2l}\]
		We now estimate the second term in \Cref{kernelest}.
		\begin{align*}
			H\ast K(x)\chi_{\{|\cdot|\leq 2^{10l}\}}(x)&=\chi_{\{|\cdot|\leq 2^{10l}\}}(x)\int_{\R^2}H(y)K(x-y)dy\\
			&=\chi_{\{|\cdot|\leq 2^{10l}\}}(x)\left(\int\limits_{|y|>2^{20l}}H(y)K(x-y)dy+\int\limits_{|y|\leq 2^{20l}}K_1(y)K_2(x-y)dy\right)\\
			&=\chi_{\{|\cdot|\leq 2^{10l}\}}(x)\int_{\R^2}H(y)\chi_{\{|\cdot|\leq 2^{20l}\}}(x)K(x-y)dy\\
			&~~\, \ +\chi_{\{|\cdot|\leq 2^{10l}\}}(x)\int_{|y|> 2^{20l}}H(y)K_2(x-y)dy\\
			&= \left(H\chi_{\{|\cdot|\leq 2^{20l}\}}\right)\ast K(x)\chi_{\{|\cdot|\leq 2^{10l}\}}(x)+L_2(x).
		\end{align*}
		
		For kernel $L_2$, we have the estimate
		\begin{align*}
			\left\|L_2\right\|_1&\leq \int_{|x|\leq2^{10l}}\int_{|y|> 2^{20l}}|K_1(y)K(x-y)|dydx\\
			&= \int_{|y|> 2^{20l}}|H(y)|\int_{|x|\leq2^{10l}}|K(x-y)|dxdy\\
			&= \sum_{k=20l}^{\infty}\int_{A_k}|H(y)|\int_{|x|\leq2^{10l}}|K(x-y)|dxdy\\
			&\leq \sum_{k=20l}^{\infty}\int_{A_k}|H(y)|\int_{|z|>2^{k-10l}}|K(z)|dxdz, 
		\end{align*}
		where we have used that $|z|=|x-y|>2^{k-10l}$ when $|x|\leq2^{10l}$ and $|y|> 2^{k}$. We know that $\|H\|_{L^1(A_k)}\leq C.$ Using integral estimate on $K$ as above, we get that
		\[\int_{|z|>2^{k-10l}}|K(z)|dz\leq 2^{-k+11l}.\]
		Therefore, we obtain that 
		\begin{align*}
			\|L_2\|_1 \leq \sum_{k=20l}^{\infty}C2^{k-11l}\lesssim2^{-9l}.
		\end{align*}
	This completes the proof of \Cref{kernelest}.
\end{proof}
\section{Proof of \Cref{BR}: Bilinear Bochner-Riesz means}\label{sec:proofofbr}
	Observe that in view of \Cref{smoothOmega} in~\Cref{sec:redtosmooth} and Fatou's lemma, it is enough to establish \Cref{BR} for domain $\Omega$ with $C^\infty-$smooth boundary with implied bounds depending only on the $C^1-$parametrization of the boundary $\partial\Omega$.
	
	We shall complete the proof of \Cref{BR} with an additional assumption on $\Omega$ that no portion of the boundary $\partial\Omega$ is parallel to the coordinates axes. Observe that, since $\Omega$ is a convex domain, the boundary $\partial\Omega$ can turn parallel to the coordinates axes atmost four times. This assumption will be removed at a later stage to complete the proof for general convex domains as considered in \Cref{BR}.
	
	 Let $\phi,\psi\in C^\infty_c(\R^2)$ be such that $\mathrm{supp}(\phi)\subset B(0,1)$, $\phi(x)=1$ for $x\in B(0,\frac{1}{2})$ and $\mathrm{supp}(\psi)\subset \{(\xi,\eta)\in \R^2:\frac{1}{2}\leq|(\xi,\eta)|\leq 2\}$ and 	\begin{eqnarray*}
	 (1-\rho(\xi,\eta))_+^\lambda&=&\phi(\xi,\eta)(1-\rho(\xi,\eta))_+^\lambda+\sum\limits_{l=1}^\infty2^{-\lambda l} \psi(2^l(1-\rho(\xi,\eta)))(1-\rho(\xi,\eta))_+^\lambda\\
	 &=&m_{0}(\xi,\eta)+\sum\limits_{l=1}^\infty 2^{-\lambda l} m_l(\xi,\eta).
	 \end{eqnarray*}
	We shall prove suitable estimates for each of the multiplier $m_l$ as above. We will decomposition these pieces further. The parametrization of the boundary $\partial\Omega\cap\mathfrak{G}_{u_p}$ as in \Cref{para} allows us  to decompose the multiplier $m_l$ into $2^{2M}$ pieces. Let us use the notation from \Cref{para} here. 
	
	Let $S_p$ be the sector with its bisector passing through the vector $u_p$ and having arc length $2^{-2M+1}$ on the unit circle. Let $b_{p}\in C^\infty_c(\R^2)$ be a radial function supported in $S_p$ such that 
	%$\sum\limits_{p=1}^{2^{2M}}b_{p}(\xi,\eta)=1$ for $(\xi,\eta)\neq(0,0)$. Hence we have
	\[m_l=\sum\limits_{p=1}^{2^{2M}}m_lb_p.\]
	Next, we invoke the refinement of the boundary decomposition from~  \Cref{refinement} to decompose the multiplier further. For the vector $u_p=e^{i\theta_p}$, consider the set of intervals $\{A_{l,p,j,v}:\;j=1,\dots,Q,\;\nu=-2M-l,\dots,2M+l\}$ as obtained in \Cref{refinement}. Let $I_{l,p,j,\nu}^*$  denote the union of two intervals containing $a_{j,\nu}$ and $I_{l,p,j,\nu}=\frac{2}{3}I_{l,p,j,\nu}^*$. Let $\beta_{j,v}^p\in C_c^\infty(\R)$ be the function supported in the interval $I_{l,p,j,\nu}$ such that
	\[\sum\limits_{j,\nu}\beta_{j,\nu}^p(t)=1,\;-1\leq t\leq 1,\;\text{and}\]
	\begin{equation}\label{derivative}
		\left|\frac{d^n}{dt^n}\beta_{j,\nu}^p(t)\right|\lesssim |I_{l,p,j,\nu}|^{-n},\;\text{for}\;n=1,2,3,4.
	\end{equation}
	This gives us the following decomposition of the multiplier $m_l(\xi,\eta)$.  $$m_l(\xi,\eta)=\sum\limits_{p=1}^{2^{2M}}\sum\limits_{j,\nu}m_l(\xi,\eta)b_p(\xi,\eta)\beta_{j,\nu}^p(\langle u_p^\perp,(\xi,\eta)\rangle)=:\sum\limits_{p=1}^{2^{2M}}\sum\limits_{j,\nu}m_{l,p,j,\nu}(\xi,\eta).$$
	Denote $K_{l,p,j,\nu}:=\mathcal{F}^{-1}(m_{l,p,j,\nu})$. 
	
	We can write 
	\begin{equation}\label{annulus}
		K_{l,p,j,\nu}:=\mathcal{F}^{-1}(m_{l,p,j,\nu})=\sum\limits_{k=0}^{10l}K_{l,p,j,\nu}\chi_{A_k}+\sum\limits_{k=10l+1}^{\infty}K_{l,p,j,\nu}\chi_{A_k}.
	\end{equation}
	Let us first estimate terms with $k> 10l$. We have 
	\[K_{l,p,j,\nu}(x)=\mathcal{F}^{-1}{m_l}* H_{l,p,j,\nu}(x),\]
	where $H_{l,p,j,\nu}=\mathcal{F}^{-1}(b_p(.)\beta_{j,\nu}^p(\langle u_p^\perp,(.)\rangle))$.
	
	\Cref{SZlemma} applied to the function  $h=\psi(2^l(1-t))$ yields that  $$\|\mathcal{F}^{-1}{m_l}\|_{L^1(A_i)}\leq 2^{l-\frac{i}{2}},\;i\in\N.$$ 
	The estimate \eqref{derivative} along with integration by parts argument applied to $H_{l,p,j,\nu}$ twice gives us 
	$$\left|H_{l,p,j,\nu}(x)\right| \lesssim \frac{|I|}{(1+|I||\langle x,u_p\rangle|)^2}\frac{1}{1+|\langle x,u_p^\perp\rangle|^2}.$$
	Observe that with this kernel estimate we can apply \Cref{kernelest} (part $1$) to get that  
	\begin{equation}\label{k>10l}
		\|\sum_{k=10l+1}^\infty K_{l,p,j,\nu}\chi_{A_k}\|_{L^1}\lesssim 2^{-3l}.
	\end{equation}
This decay with respect to $l$ in the estimate above, allows us to verify that the kernel is integrable. Therefore, for $p_1,p_2,p_3\geq 1,\;\frac{1}{p_1}+\frac{1}{p_2}=\frac{1}{p_3}$, we can apply \Cref{intker} to get the bilinear multiplier operator under consideration maps $L^{p_1}(\R)\times L^{p_2}(\R)$ into $L^{p_3}(\R)$ with its norm bounded by 
	\begin{align*}
		\left\|\sum_{l=1}^{\infty}2^{-\lambda l}\sum_{p=1}^{2^{2M}}\sum_{j,\nu}\sum_{k=10l}^\infty K_{l,p,j,\nu}\chi_{A_k}\right\|_{1}&\lesssim \sum_{l=1}^{\infty}2^{-(\lambda+3)l}lQ\\
		&\lesssim\sum_{l=1}^{\infty}2^{-(\lambda+3-\kappa_\Omega-\epsilon)l}l\lesssim 1,
	\end{align*}
	where we have used the fact that $Q\leq 2^{(\kappa_\Omega+\epsilon)l}$ for any $\epsilon>0$, refer to the estimate \eqref{kappa}.
	
	Therefore, we are left with estimating the bilinear operators corresponding to kernels with $k\leq10l$ in \eqref{annulus}. 
	
	Observe that
	\[\beta_{j,\nu}^p\left(\frac{\langle u_p^\perp,(\xi,\eta)\rangle}{\rho(\xi,\eta)}\right)=1,\;\text{for}\;(\xi,\eta)\in\mathrm{supp}(m_{l,p,j,\nu}).\]
	We can write  $$m_{l,p,j,\nu}(\xi,\eta)=m_l(\xi,\eta)\beta_{j,\nu}^p\left(\frac{\langle u_p^\perp,(\xi,\eta)\rangle}{\rho(\xi,\eta)}\right)b_p(\xi,\eta)\beta_{j,\nu}^p(\langle u_p^\perp,(\xi,\eta)\rangle).$$ 
	The kernel takes the form 
	$K_{l,p,j,\nu}(x)=J_{l,p,j,\nu}* H_{l,p,j,\nu}(x),$ 
	where\\ $J_{l,p,j,\nu}=\mathcal{F}^{-1}\left(m_l(\xi,\eta)\beta_{j,\nu}^p\left(\frac{\langle u_p^\perp,(\xi,\eta)\rangle}{\rho(\xi,\eta)}\right)\right)$.
	
	Let $P_{l,p,j,\nu}^1$ and $P_{l,p,j,\nu}^2$ be the projection of the support of the multiplier $m_{l,p,j,\nu}$ onto $\xi$-axis and $\eta$-axis respectively.  For each $i=1,2$, the intervals $\{P_{l,p,j,\nu}^i,\;j=1,\dots,Q_{u_p}(2^{-l})\}$ have bounded overlap, independent of the parameter $l$. Consider the Fourier projection operators defined by $\hat f_{l,p,j,\nu}=\chi_{P_{l,p,j,\nu}^1}\hat f$ and $\hat g_{l,p,j,\nu}=\chi_{P_{l,p,j,\nu}^2}\hat g$. This helps us rewrite the bilinear operator associated with kernel $K_{l,p,j,\nu}$ as follows. 
	\[K_{l,p,j,\nu}\ast (f, g)(x)=K_{l,p,j,\nu}\ast(f_{l,p,j,\nu}, g_{l,p,j,\nu})(x).\]
	Here we have used the notation 	$K\ast(f,g)(x)=K\ast (f\otimes g)(x,x).$ 
	
	In order to prove the required estimate on the kernel $J_{l,p,j,\nu}$, we introduce a homogeneous coordinate system associated with the boundary $\partial\Omega$ given by
	\[(s,\alpha)\mapsto (\xi,\eta)(s,\alpha)=s(u_p^\perp\alpha+u_p\gamma(\alpha)),\]
	where $s=\rho(\xi,\eta)$ and $\gamma$ is the map used in parametrizing the boundary as in \Cref{para}. It is easy to verify that the Jacobian of this change of variables is given by $s(\alpha\gamma'(\alpha)-\gamma(\alpha))$. Therefore, we can write the kernel $J_{l,p,j,\nu}$ as
	\[J_{l,p,j,\nu}(R_{u_p}x)=\int\limits_{s=\frac{1}{2}}^2sm_l(s)\int\limits_{\alpha=-2^{-2M-2}}^{2^{-2M-2}}\beta_{j,\nu}^p(\alpha)e^{is(\alpha x_1+\gamma(\alpha)x_2)}(\alpha\gamma'(\alpha)-\gamma(\alpha))\;d\alpha ds.\]
	Let $\eta\in C_c^\infty(\R^2)$ be such that  $\mathrm{supp}(\eta)\in B(0,2^{-2M-10})$ and $\eta=1$ on $B(0,{2^{-2M-11}})$. Define 
	\[\Phi_{0}(x,\alpha)=\phi(|I_{l,p,j,\nu}|(x_1+x_2\gamma'(\alpha)))\eta\left(\frac{x_1+x_2\gamma'(\alpha)}{|x|}\right),~\text{and}~\] 
	\[\Phi_n(x,\alpha)=\left(\phi(2^{-n-1}|I_{l,p,j,\nu}|(x_1+x_2\gamma'(\alpha)))-\phi(2^{-n}|I_{l,p,j,\nu}|(x_1+x_2\gamma'(\alpha)))\right)\eta\left(\frac{x_1+x_2\gamma'(\alpha)}{|x|}\right).\]
	We can write 
	\begin{align*}
		&J_{l,p,j,\nu}(R_{u_p}x)\\
		&=\sum\limits_{n=0}^\infty\int\limits_{s=\frac{1}{2}}^2sm_l(s)\int\limits_{\alpha=-2^{-2M-2}}^{2^{-2M-2}}\Phi_n(x,\alpha)\beta_{j,\nu}^p(\alpha)e^{is(\alpha x_1+\gamma(\alpha)x_2)}(\alpha\gamma'(\alpha)-\gamma(\alpha))\;d\alpha ds\\
		&~~~\ \ +\int\limits_{s=\frac{1}{2}}^2sm_l(s)\int\limits_{\alpha=-2^{-2M-2}}^{2^{-2M-2}}\left[1-\eta\left(\frac{x_1+x_2\gamma'(\alpha)}{|x|}\right)\right]\beta_{j,\nu}^p(\alpha)e^{is(\alpha x_1+\gamma(\alpha)x_2)}(\alpha\gamma'(\alpha)-\gamma(\alpha))\;d\alpha ds\\
		&:=\sum\limits_{n=0}^\infty J_{l,p,j,\nu}^n(R_{u_p}x)+\tilde J_{l,p,j,\nu}(R_{u_p}x)
	\end{align*}
	Observe that due to the supports of functions $\phi$ and $\eta,$ only finitely many terms in the expression above contribute non-trivially. Indeed, we need to consider only $C_M\log(1+|I_{l,p,j,\nu}||x|)$ many terms. 
	The kernel $J_{l,p,j,\nu}^n$ satisfies the following kernel estimates. 
	\begin{lemma}\label{pointwise}
		The following estimates holds true,
		\begin{align}
			|J_{l,p,j,\nu}^0(x)|\lesssim& \int\limits_{\substack{\alpha\in I_{l,p,j,\nu}^*:\\|\langle x,u_p\rangle+\langle x,u_p^\perp\rangle \gamma'(\alpha)|\label{J0}\\\leq |I_{l,p,j,\nu}|^{-1}}}\frac{2^{-l}\;d\alpha}{(1+2^{-l}|\langle x,u_p\rangle\alpha+\langle x,u_p^\perp\rangle \gamma(\alpha)|)^4},\\
			|J_{l,p,j,\nu}^n(x)|\lesssim& \int\limits_{\substack{\alpha\in I_{l,p,j,\nu}^*:\\|\langle x,u_p\rangle+\langle x,u_p^\perp\rangle \gamma'(\alpha)|\\\sim 2^n|I_{l,p,j,\nu}|^{-1}}}\frac{(1+2^{-n}|\langle x,u_p^\perp\rangle||I_{l,p,j,\nu}|)|\gamma''(\alpha)|+|I_{l,p,j,\nu}|^{-1}}{|\langle x,u_p\rangle+\langle x,u_p^\perp\rangle \gamma'(\alpha)|}\label{Jn}\\
			&\hspace{3cm}\times\frac{2^{-l}\;d\alpha}{(1+2^{-l}|\langle x,u_p\rangle\alpha+\langle x,u_p^\perp\rangle \gamma(\alpha)|)^4},\nonumber\\
			|\tilde J_{l,p,j,\nu}(x)|\lesssim& \int\limits_{I_{l,p,j,\nu}^*}\frac{2^{-l}(|I_{l,p,j,\nu}|^{-1}+|\gamma''(\alpha)|)\;d\alpha}{|\mathcal{R}_{u_p}x|(1+2^{-l}|\langle x,u_p\rangle\alpha+\langle x,u_p^\perp\rangle \gamma(\alpha)|)^4}\label{tildeJ}.
		\end{align}
	\end{lemma}
The proof of these estimates follows by an integration by parts argument in the variables $\alpha$ and $s$. We refer to \cite{SZ} for similar estimates. As a consequence of \Cref{pointwise} we obtain the following integral estimates on the kernels. 
	
	\begin{lemma}\label{integrable}
		We have the following estimates for all $n\in\N$ and $t\geq 0$,
		\begin{align}
			\|J_{l,p,j,\nu}^0\|_{1}&\lesssim 1,\label{J0int}\\
			\|J_{l,p,j,\nu}^n\|_{L^1(A_t)}&\lesssim 2^{-|t-l|},\label{Jnint}\\
			\|\tilde J_{l,p,j,\nu}\|_{L^1(A_t)}&\lesssim 1.\label{Jtildeint}
		\end{align}
		\begin{equation}\label{kerint}
			\|K_{l,p,j,\nu}\|_{1}\lesssim l.
		\end{equation}
	\end{lemma}
	\begin{proof}
		Apply the change of variables given by 
		\[v_1=2^{-l}(\langle x,u_p\rangle+\langle x,u_p^\perp\rangle \gamma'(\alpha)),\;v_2=2^{-l}(\langle x,u_p\rangle\alpha+\langle x,u_p^\perp\rangle \gamma(\alpha)).\]
		Note that the Jacobian of this map is bounded by $c2^{2l}$, therefore, we obtain that,
			\[\|J_{l,p,j,\nu}^0\|_{1}\lesssim\int\limits_{\alpha\in I_{l,p,j,\nu}^*}\iint\limits_{|v_1|\leq 2^{-l}|I_{l,p,j,\nu}|^{-1}}\frac{2^{l}\;dv_1dv_2d\alpha}{(1+|v_2|)^4}\lesssim 1.\]
		The estimate \eqref{Jnint} follows from the above change of variables argument. Indeed, we have
		\begin{equation}\label{eqnjn}
			\|J_{l,p,j,\nu}^n\|_{L^1(A_t)}\lesssim\int\limits_{\alpha\in I_{l,p,j,\nu}^*}((1+2^{t-n}|I_{l,p,j,\nu}|)|\gamma''(\alpha)|+|I_{l,p,j,\nu}|^{-1})\iint\limits_{\substack{|v_1|\sim 2^{n-l}|I_{l,p,j,\nu}|^{-1}\\(v_1,v_2)\sim2^{t-l}}}\frac{dv_1dv_2d\alpha}{|v_1|(1+|v_2|)^4}\\			
		\end{equation}
		Using \eqref{interval}, we get that $\int_{I_{l,p,j,\nu}^*}|I_{l,p,j,\nu}|\gamma''(\alpha)d\alpha\leq 2^{-l}$. Hence the integral in $\alpha$ is dominated by a multiple of $(2^{t-n-l}+1)$. Moreover, $\mathrm{supp}(J_{l,p,j,\nu}^n)\cap A_t\neq\emptyset$ implies that $2^n|I_{l,p,j,\nu}|^{-1}\leq 2^{-M}2^t$, whence $|v_2|\gtrsim 2^{t-l}$ for the domain of integration in \eqref{eqnjn}. Thus, we have
		\[\|J_{l,p,j,\nu}^n\|_{L^1(A_t)}\lesssim (2^{t-n-l}+1)\min\{2^{t-l},2^{3(l-t)}\},\]
		and the estimate \eqref{Jnint} follows. 
		
		The estimate for $\tilde J_{l,p,j,\nu}$ follows from a similar argument. We leave the details to the reader.
		
		Next, we prove \eqref{kerint}. Consider,
		\begin{align*}
			\|K_{l,p,j,\nu}\|_{1}\lesssim& \|K_{l,p,j,\nu}\chi_{|.|\geq 10l}\|_{1}\\
			&+\|J_{l,p,j,\nu}^0\|_{L^1}\|H_{l,p,j,\nu}^0\|_{1}\\
			&+\sum_{t=0}^{\infty}\sum_{n=1}^t\|J_{l,p,j,\nu}^n\|_{L^1(A_t)}\|H_{l,p,j,\nu}^0\|_{1}\\
			&+\|\tilde J_{l,p,j,\nu}*H_{l,p,j,\nu}\chi_{|.|\leq 10l}\|_{1}.
		\end{align*}
		Observe that the first term in the above is already estimated in \eqref{k>10l}. The estimates for the second and third terms follow  from \eqref{J0int} and  \eqref{Jnint} respectively along with the integrability of the kernel $H_{l,p,j,\nu}$. For the last term, we use the equation \eqref{Jtildeint} to get
		\begin{align*}
			\|\tilde J_{l,p,j,\nu}*H_{l,p,j,\nu}\chi_{|.|\leq 10l}\|_{1}&\lesssim \sum_{t=0}^{20l}\|\tilde J_{l,p,j,\nu}\|_{L^1(A_t)}\|H_{l,p,j,\nu}\|_{1}+\sum_{t=20l}^\infty\|\tilde J_{l,p,j,\nu}\|_{L^1(A_t)}\|H_{l,p,j,\nu}\|_{L^1(|.|\geq 2^k)}\\
			&\lesssim l+\sum_{t=20l}^\infty2^{l-t}\lesssim l.
		\end{align*}
	\end{proof}
With these estimates on kernels, we can show that the corresponding bilinear multiplier operators can be controlled in pointwise manner by bilinear Kakeya maximal function with eccentricity depending on the parameter $l$. More precisely, we have that 
	\begin{lemma}\label{domKakeya}
		Let $n\in\N\cup\{0\}$ and $0\leq k\leq20l$. The following pointwise domination holds
		\begin{align}
			\left|(J_{l,p,j,\nu}^0\chi_{\{|\cdot|\leq 2^{20l}\}})*(f, g)(x)\right|&\lesssim \mathcal{M}_{\mathcal R_{\leq 30l}}(f,g)(x)\label{J0Kakeya}\\
			\left|(J_{l,p,j,\nu}^n\chi_{A_k})*(f, g)(x)\right|&\lesssim \mathcal{M}_{\mathcal R_{\leq 30l}}(f,g)(x),\label{JnKakeya}\\
			\left|(\tilde J_{l,p,j,\nu}\chi_{A_k})*(f,g)(x)\right|&\lesssim \mathcal{M}_{\mathcal R_{\leq 30l}}(f,g)(x).\label{tildeJKakeya}
		\end{align}
	\end{lemma}
	\begin{proof}
		{\bf Proof of \eqref{J0Kakeya}:} Recall the estimate \eqref{tangent} and observe that it implies that the angle between the tangent vector $(u_p+\gamma'(\alpha)u_p^\perp)$ and the perpendicular vector to the position vector $(\alpha u_p+\gamma(\alpha)u_p^\perp)$ is less than a constant, say $C_M$ which is smaller than $\frac{\pi}{2}$. Therefore, the parallelogram spanned by the vectors $(\alpha u_p+\gamma(\alpha)u_p^\perp)$ and $u_p+\gamma'(\alpha)u_p^\perp$ can be dominated by a rectangle of comparable area with sides parallel to $(\alpha u_p+\gamma(\alpha)u_p^\perp)$ and $(-\gamma(\alpha)u_p+\alpha u_p^\perp)$. This observation along with the estimate \eqref{J0} allows us to deduce that
		\[\left|(J_{l,p,j,\nu}^0\chi_{\{|\cdot|\leq 2^{20l}\}})(x)\right|\lesssim\int_{I_{l,p,j,\nu}^*}|I_{l,p,j,\nu}|^{-1}\sum_{t=0}^{20l}\frac{2^{-3t}}{|R_\alpha|}\chi_{R_\alpha}(x)d\alpha,\]
		where $R_\alpha$ is the rectangle with sides parallel to $(\alpha u_p+\gamma(\alpha)u_p^\perp)$ and $(-\gamma(\alpha)u_p+\alpha u_p^\perp)$ and  side-lengths $2^{t+l}$ and $|I_{l,p,j,\nu}|^{-1}$ respectively. Since $2^{-l}\leq |I_{l,p,j,\nu}|$, the estimate \eqref{J0Kakeya} follows. \\
		
		\noindent
		{\bf Proof of \eqref{JnKakeya}:} This involevs the kernel $J_{l,p,j,\nu}^n$ which satisfies the estimate given by \eqref{Jn}. We decompose the integral over $I_{l,p,j,\nu}^*$ in~\eqref{Jn} into two pieces given by $I^1=\{\alpha\in I_{l,p,j,\nu}^*:\;|\langle x,\alpha u_p+\gamma(\alpha)u_p^\perp\rangle|\geq |\langle x,u_p+\gamma'(\alpha)u_p^\perp\rangle|\}$ and $I^2=I_{l,p,j,\nu}^*\setminus I^1$. This gives us 
		\[\left|(J_{l,p,j,\nu}^n\chi_{A_k})(x)\right|\leq \mathfrak{I}_1(x)+\mathfrak{I}_2(x)\]
		where $\mathfrak{I}_i(x)$ corresponds to the equation~ \eqref{Jn} with integral over $I^i$ for $i=1,2.$
	
			Note that for $\mathfrak{I}_1(x)$, we have
			\begin{equation*}
				\mathfrak{I}_1(x)\lesssim\int_{I^1}((1+2^{k-n}|I_{l,p,j,\nu}|)|\gamma''(\alpha)|+|I_{l,p,j,\nu}|^{-1})\frac{\min\{2^{k-l},2^{3(l-k)}\}}{|R_\alpha|}\chi_{R_\alpha}(x)d\alpha,
			\end{equation*}
			where $R_\alpha$ is the rectangle with sides parallel to $(\alpha u_p+\gamma(\alpha)u_p^\perp)$ and $(-\gamma(\alpha)u_p+\alpha u_p^\perp)$ and the corresponding side-lengths given by a constant multiples of $2^k$ and $2^{n}|I_{l,p,j,\nu}|^{-1}$ respectively. In particular, the eccentricity of $R_\alpha$ is bounded by a constant multiple of $2^{30l}$. Also, we have that  $$\int_{I_{l,p,j,\nu}^*}|I_{l,p,j,\nu}|\gamma''(\alpha)d\alpha\leq 2^{-l}.$$
			Therefore, the bilinear operator corresponding to the kernel $\mathfrak{I}_1$ can be easily dominated by the bilinear Kakeya maximal function $\mathcal{M}_{\mathcal R_{\leq 30l}}.$
			
			Next, consider the term $\mathfrak{I}_2(x)$ which involves integral over $I^2=I_{l,p,j,\nu}^*\setminus I^1$. Note that for $\alpha\in I^2$, we have $2^n|I_{l,p,j,\nu}|^{-1}\sim|\langle x,u_p+\gamma'(\alpha)u_p^\perp\rangle|\sim|x|\sim 2^k$. In particular, $k\leq n+l$. Thus,
			\begin{align*}
			\mathfrak{I}_2(x)&\lesssim \int_{I^2}((1+2^{k-n}|I_{l,p,j,\nu}|)|\gamma''(\alpha)|+|I_{l,p,j,\nu}|^{-1})\sum_{i=1}^k\min\{2^{i-l},2^{3(l-i)}\}\frac{1}{|R_\alpha|}\chi_{R_\alpha}(x)d\alpha,
			\end{align*}
			where $R_\alpha$ is the rectangle with sides parallel to $(\alpha u_p+\gamma(\alpha)u_p^\perp)$ and $(-\gamma(\alpha)u_p+\alpha u_p^\perp)$ and side-lengths given by a constant multiples of $2^i$ and $2^k$ respectively. This proves the required estimate. \\
	
			\noindent
		{\bf Proof of \eqref{tildeJKakeya}} follows using the arguments as in the case of $\mathfrak{I}_2(x)$. 
\end{proof}
Recall that we are chasing $L^p-$boundedness of bilinear operators associated with kernels given in \Cref{annulus} for terms $k\leq 10l$. 
We note that the required estimates for the terms $\sum_{k=0}^{10l}((J_{l,p,j,\nu}\chi_{|x|\geq 2^{20l}})*H_{l,p,j,\nu})\chi_{A_k}$ follows by employing \Cref{integrable} and \Cref{kernelest}. For the remaining terms, we observe that the quantity $H_{l,p,j,\nu}\ast(f,g)$ can be dominated by product of Hardy-Littlewood maximal function. Thus, using \Cref{domKakeya}, we have
	\begin{align*}
		&\left\|\sum_{l=1}^{\infty}2^{-\lambda l}\sum_{p=1}^{2^{2M}}\sum_{j,\nu} ((J_{l,p,j,\nu}\chi_{|\cdot|\leq 2^{20l}})*H_{l,p,j,\nu})\chi_{\{|\cdot|\leq2^{10l}\}}*(f_{l,p,j,\nu}, g_{l,p,j,\nu})\right\|_{p_3}\\
		\lesssim&\sum_{l=1}^{\infty}2^{-\lambda l}\sum_{p=1}^{2^{2M}}\left\|\sum_{j,\nu}\sum_{k=0}^{20l} (J_{l,p,j,\nu}\chi_{A_k}*(Mf_{l,p,j,\nu}, Mg_{l,p,j,\nu})\right\|_{p_3}\\
		\lesssim&\sum_{l=1}^{\infty}2^{-\lambda l}\sum_{p=1}^{2^{2M}}\left\|\sum_{j,\nu}\sum_{k=0}^{20l} (J_{l,p,j,\nu}\chi_{A_k}*(Mf_{l,p,j,\nu}, Mg_{l,p,j,\nu})\right\|_{p_3}\\
		\lesssim&\sum_{l=1}^{\infty}2^{-\lambda l}\sum_{p=1}^{2^{2M}}\left\|\sum_{j,\nu}\sum_{k=0}^{20l}l \mathcal{M}_{2^{30l}}(Mf_{l,p,j,\nu},Mg_{l,p,j,\nu})\right\|_{p_3}\\
		\lesssim&\sum_{l=1}^{\infty}2^{-\lambda l}l^2\sum_{p=1}^{2^{2M}}\left\|\sum_{j,\nu} \mathcal{M}_{2^{30l}}(Mf_{l,p,j,\nu},Mg_{l,p,j,\nu})\right\|_{p_3}\\
	\end{align*}
	Now by an application of the vector valued boundedness of Kakeya maximal function (\Cref{vector}) with $\epsilon=\frac{\lambda}{2}$ and that of Hardy-Littlewood maximal function (see Theorem 5.6.6. in \cite{Gra2}), the above term can be dominated by
	\begin{align*}
		&\sum_{l=1}^{\infty}2^{-\frac{\lambda l}{2}}l^2\sum_{p=1}^{2^{2M}} \left\|\left(\sum_{j,\nu}|Mf_{l,p,j,\nu}|^2\right)^\frac{1}{2}\right\|_{p_1}\left\|\left(\sum_{j,\nu}|Mg_{l,p,j,\nu}|^2\right)^\frac{1}{2}\right\|_{p_2}\\
		\lesssim&\sum_{l=1}^{\infty}2^{-\frac{\lambda l}{2}}l^2\sum_{p=1}^{2^{2M}} \left\|\left(\sum_{j,\nu}|f_{l,p,j,\nu}|^2\right)^\frac{1}{2}\right\|_{p_1}\left\|\left(\sum_{j,\nu}|g_{l,p,j,\nu}|^2\right)^\frac{1}{2}\right\|_{p_2}\\
		\lesssim&\sum_{l=1}^{\infty}2^{-\frac{\lambda l}{2}}l^3\|f\|_{p_1}\|g\|_{p_2}\lesssim\|f\|_{p_1}\|g\|_{p_2},
	\end{align*}
where we have used the Rubio de Francia's Littlewood-Paley inequality~{\cite{Rubio}} for the collection of boundedly overlapping intervals $\{P_{l,p,j,\nu}^i,\;j=1,\dots,Q_{u_p}(2^{-l})\}$ in the second inequality.
	
This completes the proof of \Cref{BR} with the assumption that the no part of the boundary $\partial\Omega$ is parallel to coordinate axes. This assumption is easy to get around. For, let us consider the case when a portion of the boundary is parallel to a coordinate axis. Observe that we can  decompose the bilinear multiplier into ``annulus" as before to obtain $m=\sum_{l=0}^\infty m_l$. Next, we consider a smooth decomposition of each $m_l$ as $m_l=m_l^1+m_l^2$ where $m_l^1$ is supported in the union of atmost four rectangles parallel to the axes. Since the symbol $m_l^1$ is adapted to a rectangle we can use the Hilbert transform in $\xi$ and $\eta$ variables separately to deduce boundedness of the bilinear operator corresponding to $m_l^1$. Finally, the case of $m_l^2$ is dealt with similarly as above. This completes the proof of \Cref{BR}. 
\qed 

\begin{remark}
We remark here that $L^2(\R)\times L^2(\R)\to L^1(\R)-$boundedness of the operator  $\mathcal{B}^\lambda, \lambda>0,$ can be deduced with a simpler argument. Indeed, the estimates \eqref{kerint} and \Cref{intker} imply that the operator associated to the multiplier $m_{l,p,j,\nu}$ maps $L^2(\R)\times L^2(\R)$ into $L^1(\R)$ with operator norm controlled by $l$. Therefore, a simple use of Cauchy-Schwartz inequality yields the desired boundedness result. For, consider 
	\begin{align*}
		\|\mathcal{B}^\lambda(f,g)\|_{1}&\lesssim\sum_l 2^{-\lambda l}l\sum_p\sum_{j,\nu}\|f_{l,p,j,\nu}\|_{2}\|g_{l,p,j,\nu}\|_{2}\\
		\lesssim&\sum_{l=1}^{\infty}2^{-\lambda l}l\sum_{p=1}^{2^{2M}}\left\|\left(\sum_{j,\nu}|f_{l,p,j,\nu}|^2\right)^\frac{1}{2}\right\|_{2}\left\|\left(\sum_{j,\nu}|g_{l,p,j,\nu}|^2\right)^\frac{1}{2}\right\|_{2}\\
		\lesssim&\sum_{l=1}^{\infty}2^{-\lambda l}l^2\|f\|_{2}\|g\|_{2}\lesssim\|f\|_{2}\|g\|_{2}.
	\end{align*} 
\end{remark}

\section{Proof of \Cref{KN}: Fixed scale bilinear Kakeya maximal function}\label{sec:proofofkakeya1}
First, note that an easy observation using standard dilation argument we can get that for a triplet $(p_1,p_2,p_3)$ satisfying the H\"{o}lder relation $\frac{1}{p_3}=\frac{1}{p_1}+\frac{1}{p_2}$, we have that  
	\begin{equation*}
		\|\mathcal{M}_{\mathcal R_{1,N}}\|_{L^{p_1}\times L^{p_2}\rightarrow L^{p_3}} = \|\mathcal{M}_{\mathcal R_{\delta,N}}\|_{L^{p_1}\times L^{p_2}\rightarrow L^{p_3}}.
	\end{equation*}
Therefore, we only need to prove \Cref{KN} for $\delta=1$. However, in the next estimate we work with arbitrary $\delta>0$, as it will be used later in the paper in this form. \\
\noindent
{\bf Proof of Banach case part (a):}
Observe that it is enough to prove the following two estimates for a given rectangle $R\in \mathcal R_{\delta,N}$.
\begin{enumerate}
		\item For $1<s<\infty,$ we have  that	\begin{eqnarray}\label{dominationbyproduct1}
			\frac{1}{|R|}\int\limits_R |f(x-y_1)||g(x-y_2)|\;dy_1dy_2 &\lesssim & M_sf(x)M_{s'}g(x).
		\end{eqnarray}
		\item For $s=1$, we have that
		\begin{eqnarray}\label{dominationbyproduct2}\frac{1}{|R|}\int\limits_R |f(x-y_1)||g(x-y_2)|\;dy_1dy_2\lesssim\min\{ \|g\|_\infty Mf(x),\|f\|_\infty Mg(x)\}.
			\end{eqnarray}
	\end{enumerate}
	with the implicit constants in both the inequalities above independent of $R$. Here we have used the notation $M_sf(x)=(M(f^s)(x))^\frac{1}{s}, ~s>0.$ 

	Let us assume that the longest side of $R$ makes an angle $\theta_0$ with $x$-axis. Due to the symmetry in $f$ and $g$, we can without loss of generality, assume that $0<\theta_0\leq\frac{\pi}{4}$. We express the rectangle $R$ as a star shaped set   $R=\{(t\cos\theta,t\sin\theta):\;\theta\in[0,\frac{\pi}{2}]\cup[\pi,\frac{3\pi}{2}],\;0\leq t<r(\theta)\}$, where $r(\theta)$ is the function of the boundary of $R$ with respect to $\theta$. By expressing the average over $R$ in polar coordinates, we have
	\begin{align*}
		&\frac{1}{|R|}\int\limits_R |f(x-y_1)||g(x-y_2)|\;dy_1dy_2\\
		\leq&\frac{1}{\delta^2 N}\int\limits_{\theta=0}^{\frac{\pi}{2}}\int\limits_{r=0}^{r(\theta)}|f(x-t\sin\theta)||g(x-t\cos\theta)|t\;dtd\theta\\
		\leq&\frac{1}{\delta^2 N}\int\limits_{\theta=0}^{\frac{\pi}{2}}r^2(\theta)\left(\frac{1}{r(\theta)}\int\limits_{r=0}^{r(\theta)}|f(x-t\sin\theta)|^s\;dt\right)^\frac{1}{s}\left(\frac{1}{r(\theta)}\int\limits_{r=0}^{r(\theta)}|g(x-t\cos\theta)|^{s'}\;dt\right)^\frac{1}{s'}d\theta\\
		\leq&\left(\frac{1}{\delta^2 N}\int\limits_{\theta=0}^{\frac{\pi}{2}}r^2(\theta)d\theta\right)M_sf(x)M_{s'}g(x)\\
		\lesssim&\frac{1}{\delta^2 N}\left(\int\limits_{\theta=0}^{\theta_0-\frac{C}{N}}(\delta\cosec(\theta_0-\theta))^2
		d\theta+\int\limits_{\theta=\theta_0-\frac{C}{N}}^{\theta_0+\frac{C}{N}}(\delta N)^2\;d\theta+\int\limits_{\theta=\theta_0+\frac{C}{N}}^{\frac{\pi}{2}}(\delta\cosec(\theta-\theta_0))^2d\theta\right)M_sf(x)M_{s'}g(x)\\
		\lesssim&M_sf(x)M_{s'}g(x). 
	\end{align*}
	where we have used the fact that $r(\theta)\lesssim \delta N$ when $\theta$ is the angle between $x$-axis and the line passing through origin and a point on the shorter side of $R$. In the remaining cases we have $r(\theta)\sim \delta|\cosec (\theta-\theta_0)|$.
	This completes the proof of the first inequality. The proof of the other inequality may be completed in the same manner. \\
	\noindent
{\bf Proof of Banach case part (b):}
%%%%%%
	%%%%%%%%%%%%
	
	For $i\in\Z$, let $I_i$ denote the interval $[i-\frac{1}{2},i+\frac{1}{2})$. Write $\R=\bigcup\limits_{i\in\Z} I_i$. 
	
	By the local integrability of $f$ and $g$, we can find for every interval $I_i$, a rectangle $R_i \in \mathcal R_{1,N}$ such that
	\begin{enumerate}
		\item $\{(x,x),x\in I_i\} \cap R_i \neq \emptyset$,
		\item $\mathcal{M}_{\mathcal R_{1,N}}(f,g)(x) \leq \frac{2}{\left|R_i\right|} \int_{R_i} f(y_1)g(y_2) dy_1dy_2, \quad \forall x \in I_i$. 
	\end{enumerate}
Let $e_i=(e_{i,1},e_{i,2})$ denote the unit vector parallel to the longest side of $R_i$. We organize the rectangles $R_i$ into three collection with the help of following sets. 
	\begin{eqnarray*}
		&&A_1=\{i\in\Z\mid\frac{1}{\sqrt{2}}< |e_{i,1}|\leq 1\},\\
		&&A_2=\{i\in\Z\mid0\leq |e_{i,1}|< \frac{1}{2}\},\\
		&\text{and}&A_3=\{i\in\Z\mid\frac{1}{2}\leq |e_{i,1}|\leq  \frac{1}{\sqrt{2}}\}.
	\end{eqnarray*}
	Let $Q_j=J_{1,j}\times J_{2,j}$ be the square in $\R^2$, where $J_{1,j}=(j_1-\frac{1}{2}, j_1+\frac{1}{2})$ and $J_{2,j}=(j_2-\frac{1}{2}, j_2+\frac{1}{2}),~j=(j_1,j_2) \in \Z^2.$ Define
	$$\gamma_i=\{j \in \Z^2\mid Q_j \cap R_i \neq \emptyset\}.$$
	The following lemma quantifies the intersection of the rectangles $R_i'$s  when projected onto the coordinate axes. 
	\begin{lemma}(Key lemma)\label{counting}
		Let $h_{l,k}(y_l),l=1,2;k=1,2,3,$ be functions on $\R$ defined as
		$$h_{l,k}(y_l)=\sum_{i\in A_k}\sum_{j \in \gamma_i}\chi_{J_{l,j}}(y_l).$$
		Then we have,
		\begin{enumerate}
		\item $\|h_{l,k}\|_{\infty} \lesssim N\log N$ for $l,k=1,2$ with $l\neq k.$
		\item $\|h_{l,l}\|_{\infty} \lesssim N$ for $l=1,2$.
		\item $\|h_{l,3}\|_{\infty} \lesssim N$ for $l=1,2.$
		\end{enumerate}
	\end{lemma}
	Let us assume \Cref{counting} for the moment and complete the proof of \Cref{KN}. Consider 
	\begin{align*}
	\mathcal{M}_{\mathcal R_{1,N}}(f,g)(x) & \nonumber \leq \sum_{i \in \Z} \frac{2}{\left|R_i\right|} \int_{R_i} f(y_1)g(y_2) dy_1dy_2 \chi_{I_i}(x)\\
	&  \nonumber =\frac{2}{N} \sum_{k=1}^{3}\sum_{i \in A_k} \int_{R_i} f(y_1)g(y_2) dy_1dy_2 \chi_{I_i}(x)\\
	&=\frac{2}{N} \sum_{k=1}^{3}\sum_{i \in A_k} \sum_{j \in \gamma_i} \int_{J_{1,j}} f(y_1)dy_1 \int_{J_{2,j}} g(y_2) dy_2.
	\end{align*}
This estimate above along with H\"{o}lder's inequality yields 
	\begin{eqnarray*}
		&&\int_{\R}|\mathcal{M}_{\mathcal R_{1,N}}(f,g)(x)|^{p_3} dx \\
		&\leq&\left(\frac{2}{N}\right)^{p_3}\sum_{k=1}^{3}\sum_{i \in A_k} \left(\sum_{j \in \gamma_i} \int_{J_{1,j}} f(y_1)dy_1 \int_{J_{2,j}} g(y_2) dy_2\right)^{p_3}\\
		%&\leq& \left(\frac{2}{N}\right)^{p_3}\sum_{k=1}^{3}\sum_{i \in A_k}\sum_{j \in \gamma_i} \left(\int_{J_{1,j}}f(y_1)dy_1\right)^{p_3}\left(\int_{J_{2,j}}g(y_2)dy_2\right)^{p_3}\\
		%&\leq& \left(\frac{2}{N}\right)^{p_3}\sum_{k=1}^{3}\sum_{i \in A_k}\left(\sum_{j \in \gamma_i} \left(\int_{J_{1,j}}f(y_1)dy_1\right)^{p_1}\right)^{\frac{p_3}{p_1}}\left(\sum_{j\in\gamma_i}\left(\int_{J_{2,j}}g(y_2)dy_2\right)^{p_2}\right)^{\frac{p_3}{p_2}}\\
		&\leq& \left(\frac{2}{N}\right)^{p_3}\sum_{k=1}^{3}\sum_{i\in A_k} \left(\sum_{j \in \gamma_i} \int_{J_{1,j}}f(y_1)^{p_1}dy_1\right)^{\frac{p_3}{p_1}}\left(\sum_{j \in \gamma_i} \int_{J_{2,j}} g(y_2)^{p_2} dy_2\right)^{\frac{p_3}{p_2}}\\
		&\leq& \left(\frac{2}{N}\right)^{p_3}\sum_{k=1}^{3} \left(\sum_{i\in A_k}\sum_{j \in \gamma_i} \int_{J_{1,j}}f(y_1)^{p_1}dy_1\right)^{\frac{p_3}{p_1}} \left(\sum_{i\in A_k}\sum_{j \in \gamma_i} \int_{J_{2,j}} g(y_2)^{p_2} dy_2\right)^\frac{p_3}{p_2}\\
		&\leq& \left(\frac{2}{N}\right)^{p_3} \sum_{k=1}^{3} \left(\int_{\R}\left(\sum_{i\in A_k}\sum_{j \in \gamma_i}\chi_{J_{1,j}}(y_1)\right)f(y_1)^{p_1}dy_1\right)^{\frac{p_3}{p_1}}\\
		&&\hspace{3cm}\times\left(\int_{\R} \left(\sum_{i\in A_k}\sum_{j \in \gamma_i}\chi_{J_{1,j}}(y_1)\right)g(y_2)^{p_2} dy_2\right)^{\frac{p_3}{p_2}}.
	\end{eqnarray*}
Invoking the estimates from \Cref{counting} we get that
	\begin{eqnarray*}
		&&\int_{\R}|\mathcal{M}_{\mathcal{R}_{1,N}}(f,g)(x)|^{p_3} dx\\
		&\lesssim& \frac{1}{N^{p_3}}\left(N^\frac{p_3}{p_1}(N\log N)^\frac{p_3}{p_2}\|f\|_{p_1}^{p_3}\|g\|_{p_2}^{p_3}+(N\log N)^\frac{p_3}{p_1}N^\frac{p_3}{p_2}\|f\|_{p_1}^{p_3}\|g\|_{p_2}^{p_3}+N^{\frac{p_3}{p_1}+\frac{p_3}{p_2}}\|f\|_{p_1}^{p_3}\|g\|_{p_2}^{p_3}\right)\\
		&\leq& N^{1-p_3}(\log N)^{\frac{p_3}{\min\{p_1,p_2\}}}\|f\|_{p_1}^{p_3}\|g\|_{p_2}^{p_3},
	\end{eqnarray*}
This completes the proof of Banach case. 
\qed

\noindent
{\bf Proofs of non-Banach case:} This part follows easily using interpolation for bilinear operators. First, observe that any rectangle $R\in \mathcal R_{1,N}$, we can dominate the bilinear average over $R$ by a bilinear average over square with its side-length comparable to $N$ and containing $R$. This gives us
\[\mathcal{M}_{\mathcal R_{1,N}}(f,g)(x)\leq NMf(x)Mg(x).\]
The H\"{o}lder's inequality along with weak-type $(1,1)$ bounds for the Hardy-Littlewood maximal operator $M$ yields the end-point result  $\|\mathcal{M}_{\mathcal R_{1,N}}\|_{L^1\times L^1\to L^{\frac{1}{2},\infty}}\lesssim N$. Finally, we obtain $(p_1,p_2,p_3)$ boundedness of $\mathcal{M}_{\mathcal R_{1,N}}$ in the non-Banach range ($\frac{1}{2}<p_3<1$) by interpolating between points $(1,\infty,1), (1,1,\frac{1}{2})$ and $(\infty,1,1)$. Note that we get the constant bounded by $N^{\frac{1}{p_3}-1}.$ This completes the proof of \Cref{KN} modulo \Cref{counting}, whose proof is given in the next section.\qed
\subsection{Proof of Key~\Cref{counting}}\label{sec:proofkeylemma}
	By the definition of $h_{l,k}$, we know that it is constant on $Q_j$. Therefore, it is enough to show that
	$$h_{l,k}(0) \leq C N\log N$$
	for sufficiently large $N$, where $C$ is a constant independent of the choice of $R_i$. 
	Let $\Gamma_1=[-\frac{1}{2},\frac{1}{2})\times \R$ and $\Gamma_2= \R \times [-\frac{1}{2},\frac{1}{2})$.
	Then,  
	$$\sum_{j \in \gamma_i}\chi_{J_{1,j}}(0)=\operatorname{card}\left(\left\{j \in \Z^2 | Q_j \cap(\Gamma_1 \cap R_i) \neq \emptyset\right\}\right)$$ and
	$$\sum_{j \in \gamma_i}\chi_{J_{2,j}}(0)=\operatorname{card}\left(\left\{j \in \Z^2 | Q_j \cap(\Gamma_2 \cap R_i) \neq \emptyset\right\}\right).$$
	
	Note that we need to consider rectangles $R_i$ which intersect either $\Gamma_1$ or $\Gamma_2$ and the maximum side-length of rectangles is $N$. Thus for $y=(0,0)$, we only need to consider $i\in[-2N,2N]$. By symmetry and definition of $h_{l,k},$ we have that 
	$$h_{l,k}(0) \leq 2 \sum_{\substack{i\in A_k\\0\leq i\leq N}}\operatorname{card}\left(\left\{j \in \Z^2 | Q_j \cap(\Gamma_l \cap R_i) \neq \emptyset\right\}\right).$$
	
	Suppose $R_i\cap\Gamma_l\neq\emptyset,$ then the length of the projection of $R_i$ on the $y_l$-axis is greater than $i-1$. On the other hand the length of the projection is always less than $N|e_{i,l}|+1$. Therefore,
	\begin{equation}\label{log}
		i-1\leq N|e_{i,l}|+1, \text{\quad i.e.\quad} |e_{i,l}|^{-1}\leq \frac{N}{i-2}.
	\end{equation}
	
	Let $L$ be a line parallel to $e_i$, then the length of $\Gamma_l \cap L$ is $|e_{i,l}|^{-1}$. Since width of $R_i$ is 1, $\Gamma_l \cap R_i$ is covered by atmost $[|e_{i,l}|^{-1}]+1$ segments of length 1. Thus, we get the following estimate from \eqref{log},
	
	$$\operatorname{card}\left(\left\{j \in \Z^2 | Q_j \cap(\Gamma_l \cap R_i) \neq \emptyset\right\}\right)\leq [|e_{i,l}|^{-1}]+2\leq \frac{N}{i-2}+2\lesssim \frac{N}{i-2}.$$
	
	Also, note that if $i\in A_l$, $|e_{i,l}|^{-1}\leq \sqrt{2}$ for $l=1,2$ and for $i\in A_3$, $|e_{i,1}|^{-1},|e_{i,2}|^{-1}\leq 2$. Therefore, for $l=1,2$ we have 
	\begin{eqnarray*}
		h_{l,l}(0) &\leq& 2 \sum_{\substack{i\in A_l\\0\leq i\leq N}}\operatorname{card}\left(\left\{j \in \Z^2 | Q_j \cap(\Gamma_l \cap R_i) \neq \emptyset\right\}\right)\\
		&\leq& 2\sum_{i=0}^{N} 3\lesssim N,
	\end{eqnarray*}
	and
	\begin{eqnarray*}
		h_{l,3}(0) &\leq& 2 \sum_{\substack{i\in A_3\\0\leq i\leq N}}\operatorname{card}\left(\left\{j \in \Z^2 | Q_j \cap(\Gamma_l \cap R_i) \neq \emptyset\right\}\right)\\
		&\leq& 2\sum_{i=0}^{N} 4\lesssim N,
	\end{eqnarray*}
	When $l\neq k$ and $k\neq3$, we obtain 
	\begin{eqnarray*}
		&&h_{l,k}(0)\\
		&\leq& 2 \sum_{\substack{i\in A_k\\0\leq i\leq N}}\operatorname{card}\left(\left\{j \in \Z^2 | Q_j \cap(\Gamma_l \cap R_i) \neq \emptyset\right\}\right)\\
		&\leq& 2\left( \sum_{i=0}^{2} \operatorname{card}\left(\left\{j \in \Z^2 | Q_j \cap(\Gamma_l \cap R_i) \neq \emptyset\right\}\right)+\sum_{\substack{i\in A_k\\3\leq i\leq N}} \operatorname{card}\left(\left\{j \in \Z^2 | Q_j \cap(\Gamma_l \cap R_i) \neq \emptyset\right\}\right)\right)\\
		&\leq& 2[3(N+1)+N\log N]\\
		&\lesssim& N\log N
	\end{eqnarray*}
This completes the proof. 
\qed
\section{Proof of \Cref{MN}: Bilinear Kakeya maximal function}\label{sec:proofofkakeya2}
The proof of \Cref{MN} for Banach case part (a) and non-Banach case can be completed using the corresponding arguments as done in the proof of \Cref{KN}. Therefore, we only need to prove Banach case part (b). Let us list down some of the estimates from these cases as we will require them to prove Banach case part (b).

We have the estimate 
\[\|\mathcal{M}_{\mathcal R_N}\|_{L^{p_1}\times L^{p_2}\to L^{p_3}}\lesssim A,\]
for triplets $(p_1,p_2,p_3)$ in each of the cases below. 
\begin{itemize}
	\item $(p_1,p_2,p_3)=(\frac{3s}{s+2},\frac{3s'}{s'+2},\frac{3}{4})$ with  $A=N^\frac{1}{3}$.
	\item $(p_1,p_2,p_3)=(s,\frac{3s'}{s'+2},\frac{3s}{3s+1})$ with $A=N^\frac{1}{3s}$.
	\item $(p_1,p_2,p_3)=(\frac{3s}{s+2},s',\frac{3s'}{3s'+1})$ with  $A=N^\frac{1}{3s'}$.
\end{itemize}
C\'ordoba \cite{Co} and Str\"omberg \cite{Str1} used an interpolation idea to deduce the logarithmic bounds in $L^2-$estimate for the linear Kakeya maximal function. We develop an appropriate bilinear analogue of the same to prove our result. We state the interpolation result as a lemma. This may be of independent  interest. The proof of Banach case part (b) follows immediately by using this interpolation lemma with the $L^p-$estimates mentioned as above. 
\begin{lemma}\label{inter}
	Let $1<s<\infty$. Suppose $T$ is a bi-sublinear operator satisfying
	\[\|T\|_{L^{p_1}\times L^{p_2}\to L^{p_3,\infty}}\lesssim A,\] 
	for the following H\"older indices $(p_1,p_2,p_3)$:
	\begin{enumerate}
		\item $(\infty,\infty,\infty)$, $(\infty,s',s')$, $(s,\infty,s)$, $(s,s',1)$, $(\infty,\frac{3s'}{s'+2},\frac{3s'}{s'+2})$, and $(\frac{3s}{s+2},\infty,\frac{3s}{s+2})$ with $A=1$.
		\item $(\frac{3s}{s+2},\frac{3s'}{s'+2},\frac{3}{4})$ with $A=N^\frac{1}{3}$.
		\item $(s,\frac{3s'}{s'+2},\frac{3s}{3s+1})$ with $A=N^\frac{1}{3s}$.
		\item  $(\frac{3s}{s+2},s',\frac{3s'}{3s'+1})$ with $A=N^\frac{1}{3s'}$.
	\end{enumerate}
	Then, we have the following strong type estimate,
	\[\|T\|_{L^s\times L^{s'}\to L^{1}}\lesssim \log N.\]
\end{lemma}
\noindent
{\bf Proof of \Cref{inter}:} We describe here the proof only for the case of $(2,2,1)$ boundedness which corresponds to $s=2$. The case of other values of $s$ may be completed as indicated at the end of this proof. 

Let $f,g\in L^2(\R)$ and $\lambda>0$. Without loss of generality we assume $\|f\|_2=\|g\|_2=1$.  Decompose $f$ as
	\[f=f_1+f_2+f_3,\text{ where}\]
	\[f_1=f\chi_{|f(x)|<\frac{\lambda^\frac{1}{2}}{4}},\;\;f_2=f\chi_{\frac{\lambda^\frac{1}{2}}{4}<|f(x)|\leq N\lambda^{\frac{1}{2}}},\;\;f_3=f\chi_{|f(x)|\geq N\lambda^{\frac{1}{2}}}.\]
	Similarly, we write $g=g_1+g_2+g_3$. Consider 
	\begin{align*}
	\|T(f,g)\|_1&=\int_0^\infty |\{x\in\R:|T(f,g)(x)|>\lambda\}|\;d\lambda\\
	&\leq \sum\limits_{i,j=1}^{3} \int_0^\infty \left|\left\{x\in\R:|T(f_i,g_j)(x)|>\frac{\lambda}{9}\right\}\right|\;d\lambda.
	\end{align*}
	By symmetry it is enough to estimate the terms when $i\leq j$. The level set of the term $i=j=1$ is of measure zero as $T$ is bounded from $L^\infty(\R)\times L^\infty(\R)\to L^\infty(\R)$. By the $L^\infty(\R)\times L^2(\R)\to L^2(\R)-$boundedness of $T$, we have
	\begin{align*}
	\int_0^\infty \left|\left\{x\in\R:|T(f_1,g_2)(x)|>\frac{\lambda}{9}\right\}\right|\;d\lambda&\lesssim\int_0^\infty \frac{1}{\lambda^2} \|f_1\|_\infty^2 \|g_2\|_2^2\;d\lambda\\
	&\lesssim\int_0^\infty\frac{1}{\lambda}\int_{\frac{1}{4}\lambda^{\frac{1}{2}}\leq|g|\leq N\lambda^\frac{1}{2}}|g(y_2)|^2\;dy_2d\lambda\\
	&\lesssim\int_{\R}\int_{\lambda=\frac{|g(y_2)|^2}{N^2}}^{16|g(y_2)|^2}\frac{d\lambda}{\lambda}|g(y_2)|^2\;dy_2\\
	&\lesssim \log N.
	\end{align*}
	For $i=j=2$, we use the $L^2(\R)\times L^2(\R)\to L^{1,\infty}(\R)-$boundedness of $T$ and Cauchy-Schwartz inequality to obtain,
	\begin{align*}
	&\int_0^\infty \left|\left\{x\in\R:|T(f_2,g_2)(x)|>\frac{\lambda}{9}\right\}\right|\;d\lambda\\
	\lesssim&\int_0^\infty \frac{1}{\lambda} \|f_2\|_2 \|g_2\|_2\;d\lambda\\
	\leq&\int_0^\infty\left(\frac{1}{\lambda}\int_{\frac{1}{4}\lambda^{\frac{1}{2}}\leq|f|\leq  N\lambda^\frac{1}{2}}|f(y_1)|^2\;dy_1\right)^\frac{1}{2}\left(\frac{1}{\lambda}\int_{\frac{1}{4}\lambda^{\frac{1}{2}}\leq|g|\leq  N\lambda^\frac{1}{2}}|g(y_2)|^2\;dy_2\right)^\frac{1}{2}d\lambda\\
	\lesssim&\left(\int_0^\infty\frac{1}{\lambda}\int_{\frac{1}{4}\lambda^{\frac{1}{2}}\leq|f|\leq N\lambda^\frac{1}{2}}|f(y_1)|^2\;dy_1d\lambda\right)^\frac{1}{2}\left(\int_0^\infty\frac{1}{\lambda}\int_{\frac{1}{4}\lambda^{\frac{1}{2}}\leq|g|\leq N\lambda^\frac{1}{2}}|g(y_2)|^2\;dy_2d\lambda\right)^\frac{1}{2}\\
	\lesssim&\left(\int_{\R}\int_{\lambda=\frac{|f(y_1)|^2}{N^2}}^{16|f(y_1)|^2}\frac{d\lambda}{\lambda}|f(y_1)|^2\;dy_1\right)^\frac{1}{2}\left(\int_{\R}\int_{\lambda=\frac{|g(y_2)|^2}{N^2}}^{16|g(y_2)|^2}\frac{d\lambda}{\lambda}|g(y_2)|^2\;dy_2\right)^\frac{1}{2}\\
	\lesssim& \log N.
	\end{align*}
	The term corresponding to $i=1,j=3$ is estimated by using the $L^\infty(\R)\times L^\frac{3}{2}(\R)\to L^\frac{3}{2}(\R)-$boundedness of $T$ as follows,
	\begin{align*}
	\int_0^\infty \left|\left\{x\in\R:|T(f_1,g_3)(x)|>\frac{\lambda}{9}\right\}\right|\;d\lambda&\lesssim \int_0^\infty \frac{1}{\lambda^\frac{3}{2}} \|f_1\|_\infty^\frac{3}{2} \|g_3\|_\frac{3}{2}^\frac{3}{2}\;d\lambda\\
	&\lesssim \int_0^\infty\frac{1}{\lambda^\frac{3}{4}}\int_{|g|\geq N\lambda^\frac{1}{2}}|g(y_2)|^\frac{3}{2}\;dy_2d\lambda\\
	&\lesssim\int_{\R}\int_{\lambda=0}^{\frac{|g(y_2)|^2}{N^2}}\frac{d\lambda}{\lambda^\frac{3}{4}}|g(y_2)|^\frac{3}{2}\;dy_2\\
	&\lesssim 1.
	\end{align*}
	For the term with $i=2,j=3$, the $L^2(\R)\times L^\frac{3}{2}(\R)\to L^{\frac{6}{7}}(\R)$ boundedness of $T$ implies that
	\begin{align*}
	&\int_0^\infty \left|\left\{x\in\R:|T(f_2,g_3)(x)|>\frac{\lambda}{9}\right\}\right|\;d\lambda\\
	\lesssim&N^\frac{1}{7}\int_0^\infty \frac{1}{\lambda^\frac{6}{7}} \|f_2\|_2^\frac{6}{7} \|g_3\|_\frac{3}{2}^\frac{6}{7}\;d\lambda\\
	\leq&N^\frac{1}{7}\int_0^\infty\left(\frac{1}{\lambda}\int_{\frac{1}{4}\lambda^{\frac{1}{2}}\leq|f|\leq  N\lambda^\frac{1}{2}}|f(y_1)|^2\;dy_1\right)^\frac{3}{7}\left(\frac{1}{\lambda^\frac{3}{4}}\int_{|g|\geq  N\lambda^\frac{1}{2}}|g(y_2)|^\frac{3}{2}\;dy_2\right)^\frac{4}{7}d\lambda\\
	\lesssim&N^\frac{1}{7}\left(\int_{\R}\int_{\lambda=\frac{|f(y_1)|^2}{N^2}}^{16|f(y_1)|^2}\frac{d\lambda}{\lambda}|f(y_1)|^2\;dy_1\right)^\frac{3}{7}\left(\int_{\R}\int_{\lambda=0}^{\frac{|g(y_2)|^2}{N^2}}\frac{d\lambda}{\lambda^\frac{3}{4}}|g(y_2)|^\frac{3}{2}\;dy_2\right)^\frac{4}{7}\\
	\lesssim& 1.
	\end{align*}
	Finally, the term with $i=j=3$ is estimated using the $L^\frac{3}{2}(\R)\times L^\frac{3}{2}(\R)\to L^\frac{3}{4}(\R)$ bound of $T$. Indeed, by an application of Cauchy-Schwartz inequality we have,
	\begin{align*}
	&\int_0^\infty \left|\left\{x\in\R:|T(f_3,g_3)(x)|>\frac{\lambda}{9}\right\}\right|\;d\lambda\\
	\lesssim&N^\frac{1}{4}\int_0^\infty \frac{1}{\lambda^\frac{3}{4}} \|f_3\|_\frac{3}{2}^\frac{3}{4} \|g_3\|_\frac{3}{2}^\frac{3}{4}\;d\lambda\\
	\lesssim&N^\frac{1}{4}\left(\int_0^\infty\frac{1}{\lambda^\frac{3}{4}}\int_{|f|\geq N\lambda^\frac{1}{2}}|f(y_1)|^\frac{3}{2}\;dy_1d\lambda\right)^\frac{1}{2}\left(\int_0^\infty\frac{1}{\lambda^\frac{3}{4}}\int_{|g|\geq N\lambda^\frac{1}{2}}|g(y_2)|^\frac{3}{2}\;dy_2d\lambda\right)^\frac{1}{2}\\
	\lesssim&N^\frac{1}{4}\left(\int_{\R}\int_{\lambda=0}^{\frac{|f(y_1)|^2}{N^2}}\frac{d\lambda}{\lambda^\frac{3}{4}}|f(y_1)|^\frac{3}{2}\;dy_1\right)^\frac{1}{2}\left(\int_{\R}\int_{\lambda=0}^{\frac{|g(y_2)|^2}{N^2}}\frac{d\lambda}{\lambda^\frac{3}{4}}|g(y_2)|^\frac{3}{2}\;dy_2\right)^\frac{1}{2} \lesssim 1.
	\end{align*}
This completes the proof for $s=2$. The case $s\neq 2$ follows similarly. In this case we need to run the proof with the following decomposition $f=f_1+f_2+f_3$ and $g=g_1+g_2+g_3$ where
	\[f_1=f\chi_{|f(x)|<\frac{\lambda^\frac{1}{s}}{4}},\;\;f_2=f\chi_{\frac{\lambda^\frac{1}{s}}{4}<|f(x)|\leq N^{\frac{1}{s-1}}\lambda^{\frac{1}{s}}},\;\;f_3=f\chi_{|f(x)|\geq N^{\frac{1}{s-1}}\lambda^{\frac{1}{s}}},\text{ and}\]
	\[g_1=g\chi_{|g(x)|<\frac{\lambda^\frac{1}{s'}}{4}},\;\;g_2=g\chi_{\frac{\lambda^\frac{1}{s'}}{4}<|g(x)|\leq N^{s-1}\lambda^{\frac{1}{s'}}},\;\;g_3=g\chi_{|g(x)|\geq N^{s-1}\lambda^{\frac{1}{s'}}}.\]
\qed
\section{Proof of \Cref{vector}: Vector-valued extension of bilinear Kakeya maximal function}\label{sec:vector}
To prove \Cref{vector}, we will employ the arguments similar to \cite{MS} that they used to obtain similar vector valued inequalities for bilinear maximal function defined in \eqref{Lacey}. First, observe that if $\M_{\mathcal R_N}$ is bounded from $L^{p_1}(\R)\times L^{p_2}(\R)$ into $L^{p_3}(\R)$ for a H\"{o}lder related triplet $(p_1,p_2,p_3)$ with $1<p_1,p_2\leq\infty$, then it admits vector-valued extension 

\begin{eqnarray}\label{vholder}
\mathcal M_{\mathcal R_N}: L^{p_1}(l^{p_1})(\mathbb{R}) \times L^{p_2}(l^{p_2})(\mathbb{R})\rightarrow L^{p_3}(l^{p_3})(\mathbb{R})
\end{eqnarray}
 with operator norm same as that of $\M_{\mathcal R_N}$  in the scalar case. 
	
By using estimates \eqref{dominationbyproduct1} and \eqref{dominationbyproduct2} and H\"older's inequality we have that 

	\begin{equation}\label{vp1}
	\mathcal M_{\mathcal R_N}:L^{p}(l^{s})(\mathbb{R}) \times L^{\infty}(l^{\infty})(\mathbb{R})\to L^{p}(l^{s})(\mathbb{R}),
	\end{equation}
	and
	\begin{equation}\label{vp2}
	\mathcal M_{\mathcal R_N}:L^{\infty}(l^{\infty})(\mathbb{R}) \times L^{p}(l^{s})(\mathbb{R})\to L^{p}(l^{s})(\mathbb{R}),
	\end{equation}
where $1<p<\infty$ and $1<s\leq\infty.$

First, we interpolate between estimates \eqref{vp1} and \eqref{vp2} to obtain the following boundedness for $1<p_1,p_2,p_3<\infty$ and $1<r_1,r_2,r_3\leq\infty$,
	\begin{equation}\label{vbanach}
	\mathcal M_{\mathcal R_N}:L^{p_1}(l^{r_1})(\mathbb{R}) \times L^{p_2}(l^{r_2})(\mathbb{R})\to L^{p_3}(l^{r_3})(\mathbb{R}).
	\end{equation}
	
	Now for $\epsilon>0$, we interpolate between \eqref{vbanach} and \eqref{vholder} (with $p_3=\frac{1}{1+\epsilon}$ and $1<p_1,p_2\leq\infty$) to get that
	\begin{equation*}
	\left\|\left(\sum_{j}\left|\M_{\mathcal R_N}\left(f_{j}, g_{j}\right)\right|^{r_3}\right)^{\frac{1}{r_3}}\right\|_{L^{p_3}(\mathbb{R})}\lesssim N^\epsilon\left\|\left(\sum_{j}\left|f_{j}\right|^{r_{1}}\right)^{\frac{1}{r_{1}}}\right\|_{L^{p_{1}}(\mathbb{R})}\left\|\left(\sum_{j}\left|g_{j}\right|^{r_{2}}\right)^{\frac{1}{r_{2}}}\right\|_{L^{p_{2}}(\mathbb{R})},
	\end{equation*}
	where $1<p_1,p_2\leq\infty$, $1\leq p_3<\infty$ and $1<r_1,r_2\leq\infty$, $1\leq r_3\leq\infty$.
	In particular, we interpolate between $L^{\frac{2}{1+\epsilon}}(l^{\frac{2}{1+\epsilon}})(\mathbb{R}) \times L^{\frac{2}{1+\epsilon}}(l^{\frac{2}{1+\epsilon}})(\mathbb{R})\to L^{\frac{1}{1+\epsilon}}(l^{\frac{1}{1+\epsilon}})(\mathbb{R})$\} and  $L^{q_1}(l^{s_1})(\mathbb{R}) \times L^{q_2}(l^{s_2})(\mathbb{R})\to L^{q_3}(l^{s_3})(\mathbb{R})$ with $1<q_1,q_2,q_3<\infty$ and $1<s_1,s_2,s_3\leq\infty$. Observe that such triplets $(q_1,q_2,q_3)$ and $(r_1,r_2,r_3)$ exist as we have 
	\[\frac{1}{p_1}=\frac{\theta}{\frac{2}{1+\epsilon}}+\frac{1-\theta}{q_1}.\]
	Then
	\[\frac{1}{q_1}=\frac{1}{1-\theta}\left(\frac{1}{q_1}-\frac{\theta(1+\epsilon)}{2}\right).\]
	Note that we need to make sure that 
	\begin{align*}
	0&<\frac{1}{1-\theta}\left(\frac{1}{q_1}-\frac{\theta(1+\epsilon)}{2}\right)<1
		\end{align*}
	%0&<\frac{1}{q_1}-\frac{\theta(1+\epsilon)}{2}<1-\theta\\
	or equivalently, 
		\begin{align*}
	\frac{\theta(1+\epsilon)}{2}&<\frac{1}{q_1}<1-\theta+\frac{\theta(1+\epsilon)}{2}
	\end{align*}
	We can choose $\theta$ so that the condition above is satisfied. The choice of $q_2,r_1$ and $r_2$ can be made similarly. 
	\qed

\section{Examples for sharpness of constants}\label{sec:examples}
In this section, we provide examples to establish the sharpness of the dependence of norm of $\mathcal{M}_{\mathcal{R}_N}$ on the parameter $N$ in \Cref{MN}.
\begin{proposition}\label{example}
	Let $(p_1,p_2,p_3)$ be such that $\frac{1}{p_3}=\frac{1}{p_1}+\frac{1}{p_2}.$ Then the following lower bounds on the operator norm $\|\mathcal{M}_{\mathcal R_N}\|_{L^{p_1}\times L^{p_2}\to L^{p_3}}$ hold. 
	\begin{enumerate}
	\item$\|\mathcal{M}_{\mathcal R_N}\|_{L^{p_1}\times L^{p_2}\to L^{p_3}} \gtrsim N^{\frac{1}{p_3}-1}, ~\text{for}\;p_3<1.$
\item 	$\|\mathcal{M}_{\mathcal R_N}\|_{L^{p_1}\times L^{p_2}\to L^1}\gtrsim \log N.$
	\end{enumerate}
\end{proposition}

\begin{proof}
	Let $f_N(x)=x^{-\frac{2}{p_1}}\chi_{_{\{x:3\leq x\leq N\}}}(x)$ and $g_N(x)=x^{-\frac{2}{p_2}}\chi_{_{\{x:3\leq x\leq N\}}}(x)$. Note that $\|f_N\|_{p_1}=\|g_N\|_{p_2}\simeq C$. Let $6<x<N-1$ and consider the rectangle containing $(x,x)$ and of  dimensions $x-3\times \frac{x-3}{N}$ in the direction of unit vector $(\frac{1}{\sqrt{2}},\frac{1}{\sqrt{2}})$ with $(4,4)$ as mid-point of the small side. Then
	\begin{eqnarray*}
		\mathcal{M}_{\mathcal R_N}(f_N,g_N)(x)&\geq& \frac{N}{(x-3)^2}\int_{4+\frac{x-3}{2\sqrt{2}N}}^{x+1-\frac{x-3}{2\sqrt{2}N}}\int_{y-\frac{x-3}{\sqrt{2}N}}^y f_N(y)g_N(z) dzdy\\
		&\geq& \frac{N}{(x-3)^2}\int_{4+\frac{x-3}{2\sqrt{2}N}}^{x+1-\frac{x-3}{2\sqrt{2}N}}\int_{y-\frac{x-3}{\sqrt{2}N}}^y \frac{1}{y^\frac{2}{p_1}z^\frac{2}{p_2}} dzdy\\
		&\geq& \frac{N}{(x-3)^2}\int_{4+\frac{x-3}{2\sqrt{2}N}}^{x+1-\frac{x-3}{2\sqrt{2}N}}\frac{x-3}{\sqrt{2}N} \frac{1}{y^{\frac{2}{p_3}}} dy\\
		&=&\frac{1}{\sqrt{2}(x-3)(1-\frac{2}{p_3})}\left(\left(x+1-\frac{x-3}{2\sqrt{2}N}\right)^{1-\frac{2}{p_3}}-\left(4+\frac{x-3}{2\sqrt{2}N}\right)^{1-\frac{2}{p_3}}\right)\\
		&\gtrsim& \frac{1}{x}
	\end{eqnarray*}
	Therefore,
	\[\|\mathcal{M}_{\mathcal R_N}(f_N,g_N)\|_{p_3}^{p_3}\geq c\int_6^{N-1}\frac{1}{x^{p_3}}\gtrsim \begin{cases} 
	N^{1-{p_3}}, & {p_3}<1 \\
	\log N, & {p_3}=1 
	\end{cases}.\]
	This completes the proof.
\end{proof}
\subsection{Remarks on (linear) Kakeya maximal operator acting on  product type functions}\label{product}
In this section, we construct examples to show that the norm dependence of the linear Kakeya maximal functions in \eqref{Cordoba} and \eqref{Stromberg} on the parameter $N$ is sharp even when we restrict the class of functions to the family of functions of product type. More precisely, we have the following,
\begin{theorem}
	The following lower bounds holds for the operators $M_{\mathcal R_{1,N}}$ and $M_{\mathcal R_N}$ acting on product type functions.
	\begin{enumerate}
\item There exists a function of the form $f(x,y)=f_1(x)f_2(y)$ such that 
$\|M_{\mathcal R_{1,N}}f\|_{2} \gtrsim (\log N)^{\frac{1}{2}}\|f\|_2.$
\item There exists a function of the form $f(x,y)=f_1(x)f_2(y)$ such that 
$\|M_{\mathcal R_N}f\|_{2} \gtrsim \log N\|f\|_2.$
	\end{enumerate}
\end{theorem}
\begin{proof}
Consider the product type function  $f_N(y,z)=\frac{1}{(yz)^{\frac{1}{2}}}\chi_{1\leq y\leq N}(y)\chi_{1\leq z\leq N}(z)$. Note that $\|f_N\|_2=\log N$. Let $R_k$ be the rectangle of dimension $1\times N$ parallel to the line $z=\frac{k}{N}y$ and lying below the line $z=\frac{k}{N}y$ with $O=(0,0)$ and $P_k=\Big(\frac{N^2}{(N^2+k^2)^{\frac{1}{2}}},\frac{Nk}{(N^2+k^2)^{\frac{1}{2}}}\Big)$ as its two vertices (see \Cref{figure1}).
\begin{figure}[htp]
	\centering
	\includegraphics[width=7cm,height=7cm]{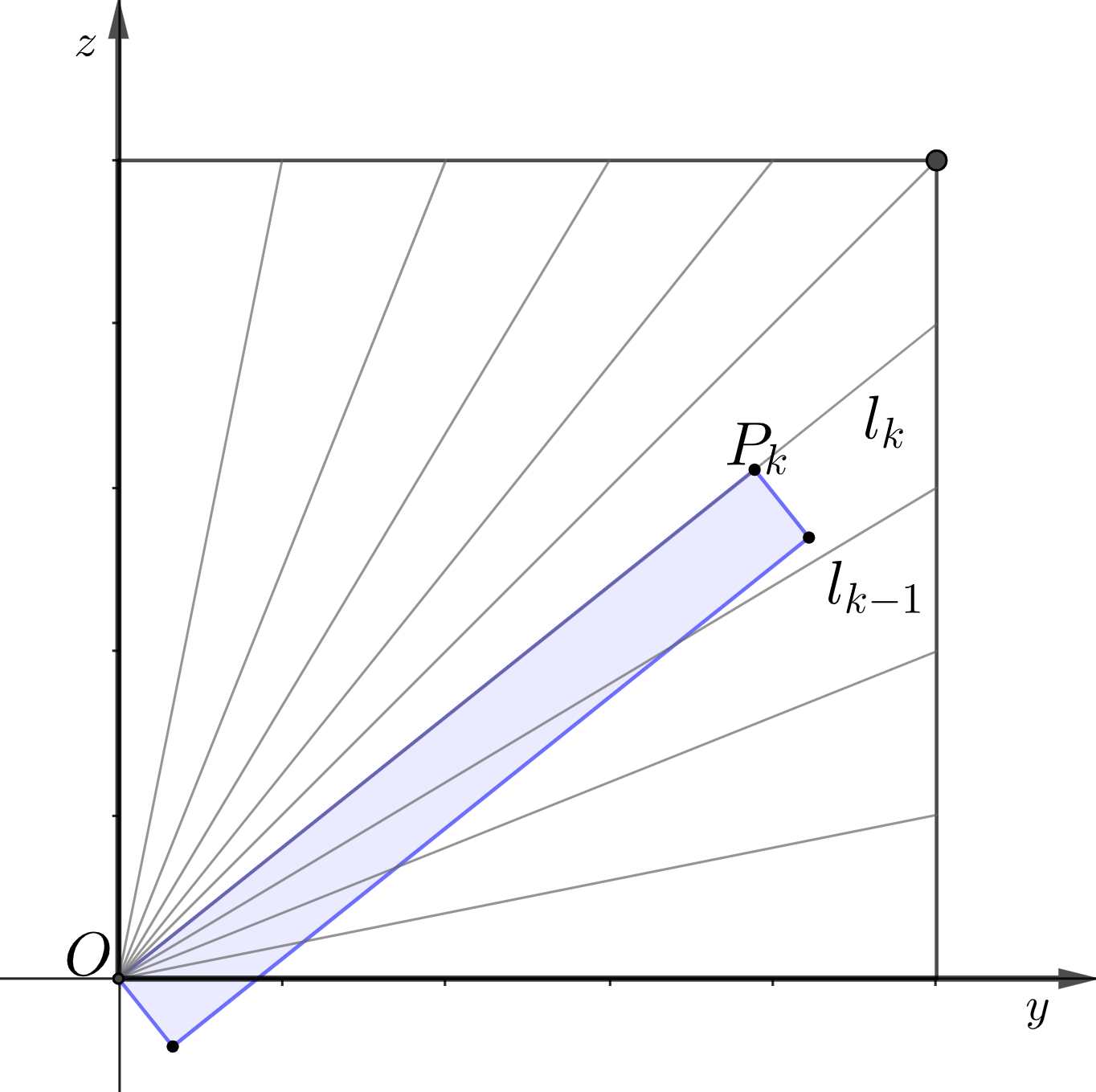}
	\caption{The blue rectangle denotes the rectangle $R_k$ with vertices $O$ and $P_k$ on the line $l_k:z=\frac{k}{N}y$.}
	\label{figure1}
\end{figure}

For $x=(x_1,x_2)\in[0,N]^2$ satisfying $\frac{k-1}{N}x_1\leq x_2< \frac{k}{N}x_1,\;k=2,3,...,N$ and $|x|\leq N$, we have
	\begin{align*}
		M_{\mathcal R_{1,N}}(f_N)(x)&\geq \frac{1}{N}\int_{R_k}f_N(y,z)dydz\\
		&\geq \frac{1}{N}\int_{\frac{N+(N^2+k^2)^{\frac{1}{2}}}{k}}^{\frac{N^2}{(N^2+k^2)^{\frac{1}{2}}}}\int_{\frac{k}{N}y-\frac{(N^2+k^2)^{\frac{1}{2}}}{N}}^{\frac{k}{N}y}\frac{N^{\frac{1}{2}}}{k^{\frac{1}{2}}y}dzdy\\
		&\geq \frac{1}{(Nk)^{\frac{1}{2}}}\frac{(N^2+k^2)^{\frac{1}{2}}}{N}\int_{\frac{2(N^2+k^2)^{\frac{1}{2}}}{k}}^{\frac{N^2}{(N^2+k^2)^{\frac{1}{2}}}}\frac{1}{y}dy\\
		&\geq \frac{1}{(Nk)^{\frac{1}{2}}}\left[\log\left(\frac{N^2}{(N^2+k^2)^{\frac{1}{2}}}\right)-\log\left(\frac{2(N^2+k^2)^{\frac{1}{2}}}{k}\right)\right]\\
		&=\frac{1}{(Nk)^{\frac{1}{2}}}\log\left(\frac{N^2k}{2(N^2+k^2)}\right)\\
		&\geq \frac{\log\frac{k}{4}}{(Nk)^{\frac{1}{2}}}.
	\end{align*}
	
	Now,
	\begin{align*}
		\|{M}_{\mathcal R_{1,N}}(f_N)\|_2^2&\geq \int_0^N\int_0^{\pi/2}|{M}_{\mathcal R_{1,N}}(f_N)(re^{i\theta})|^2rdrd\theta\\
		&\geq \int_0^N\sum_{k=2}^{N}\int_{\theta=\arctan(\frac{k-1}{N})}^{\theta=\arctan(\frac{k}{N})}\left(\frac{\log\frac{k}{4}}{(Nk)^{\frac{1}{2}}}\right)^2rdrd\theta\\
		&\gtrsim\sum_{k=2}^{N} \frac{N(\log\frac{k}{4})^2}{2k}\left(\arctan\Big(\frac{k}{N}\Big)-\arctan\Big(\frac{k-1}{N}\Big)\right)\\
		&\gtrsim (\log N)^{3}.
	\end{align*}

	Next, we take $x=(x_1,x_2)\in[1,N]^2$ such that $4\leq|x|\leq N$. For each $x$, we consider the rectangle $R_x$ containing $x$ of dimensions $\frac{|x|-2}{N}\times (|x|-2)$ with one of its shorter side touching the circle centered at origin and of radius $2$ and the longer side is parallel to the line $z=\frac{x_2}{x_1}y$ and lying below the line $z=\frac{x_2}{x_1}y$. We note that the equation of lines for longer sides of the rectangle are $z=\frac{x_2}{x_1}y$ and $z=\frac{x_2}{x_1}y-\frac{|x|(|x|-2)}{x_1N}$ and the equation of lines for shorter sides are $z=-\frac{x_1}{x_2}y+\frac{2|x|}{x_1}$ and $z=-\frac{x_1}{x_2}y+\frac{|x|^2}{x_1}$
	Thus,
	\begin{align*}
		{M}_{\mathcal R_{N}}(f_N)(x)&\geq \frac{N}{(|x|-2)^2}\int_{R_x}f_N(y,z)dydz\\
		&\geq \frac{N}{(|x|-2)^2}\int_{\frac{2x_1N+x_2(|x|-2)}{|x|N}}^{x_1}\int_{\frac{x_2}{x_1}y-\frac{|x|(|x|-2)}{x_1N}}^{\frac{x_2}{x_1}y}\frac{x_1^{\frac{1}{2}}}{x_2^{\frac{1}{2}}y}dzdy\\
		&\geq \frac{N}{(|x|-2)^2}\frac{|x|(|x|-2)}{(x_1x_2)^{\frac{1}{2}}N}\int_{\frac{2x_1N+x_2(|x|-2)}{|x|N}}^{x_1}\frac{1}{y}dy\\
		&\geq \frac{1}{(x_1x_2)^{\frac{1}{2}}}\left[\log\left(x_1\right)-\log\left(\frac{2x_1N+x_2(|x|-2)}{|x|N}\right)\right]\\
		&=\frac{1}{(x_1x_2)^{\frac{1}{2}}}\log\left(\frac{x_1|x|N}{2x_1N+x_2(|x|-2)}\right)\\
		&\geq \frac{\log\frac{|x|}{4}}{(x_1x_2)^{\frac{1}{2}}}.
	\end{align*}
	
	Therefore we have,
	\begin{align*}
		\|{M}_{\mathcal R_{N}}(f_N)\|_2^2&\geq\int_{[1,N]^2}\frac{\big(\log\frac{|x|}{4}\big)^2}{x_1x_2}\;dx\\
		&\geq\int_2^N\sum_{k=2}^{N}\int_{\theta=\arctan(\frac{k-1}{N})}^{\theta=\arctan(\frac{k}{N})}\frac{\big(\log\frac{r}{4}\big)^2}{kr^2}rdrd\theta\\
		&\geq \sum_{k=2}^{N}\left(\arctan\Big(\frac{k}{N}\Big)-\arctan\Big(\frac{k-1}{N}\Big)\right)\frac{N}{k}\int_2^N\frac{(\log \frac{r}{4})^2}{r}dr\\
		&\gtrsim (\log N)^{4}.
	\end{align*}

\end{proof}
\section{Further discussions}\label{sec:discussion}
In this section we initiate a discussion about connections of bilinear Kakeya maximal function with other type of maximal functions in the bilinear setting. The aim of this discussion is to indicate some further questions that need to be investigated. 

Let $\Omega$ denote the set of vectors in $\R^2$ and given a collection of rectangles $\mathcal R$ we use the notation $\mathcal R^{\Omega}$ to denote the collection of those rectangles in $\mathcal R$ which have their longest side parallel to some $\omega\in \Omega$.  If not stated otherwise, the elements of $\Omega$ will be unit vectors.  

For the linear case, when $\mathfrak{F}$ is the collection of rectangles with longest side parallel to one of the vectors in $\Omega$ and $\mathrm{card}(\Omega)=N$, the operator $M_\mathfrak{F}$ defined in \eqref{linearF} satisfies the bound,
\[\|M_\mathfrak{F}\|_{L^2(\R^2)\to L^2(\R^2)}\lesssim \log N.\]
The above inequality was shown to hold by Str\"omberg \cite{Str1} for the uniformly distributed set of directions and for arbitrary set of directions, this was resolved by Katz \cite{Katz}. Moreover, for the lacunary set of directions, $M_\mathfrak{F}$ was shown to be weak type $(2,2)$ by C\'ordoba and Fefferman \cite{CF}, and it was proved to be bounded in $L^p,\;1<p<\infty$ in \cite{NSW} using Fourier transform methods.

We define the directional  bilinear maximal operator 
$\mathcal M_{\mathcal R^{\Omega}}$ by 
\[\mathcal{M}_{\mathcal{R}^{\Omega}}(f,g)(x)=\sup\limits_{\substack{R\in \mathcal R^{\Omega}:(x,x)\in R}}\frac{1}{|R|}\int_R |f(y_1)||g(y_2)|\;dy_1dy_2.\]
Here we would like to refer the interested reader to Str\"omberg \cite{Str1}, C\'ordoba and Fefferman \cite{CF}, Katz \cite{Katz}, Nagel, Stein and Wainger~\cite{NSW} for some of the important results for the linear counterpart 
$$M_{\mathcal R^{\Omega}}f(x)=\sup\limits_{\substack{R\in \mathcal R^{\Omega}: x\in R}}\frac{1}{|R|}\int_R |f(y)|\;dy, ~ x\in \R^2.$$

%Observe that the following estimate holds 
%\begin{eqnarray}\label{relation1}
%\mathcal{M}_{\mathcal{R}_{1,N}}(f,g)(x)&\lesssim \mathcal{M}_{\mathcal {R}^{\Omega^N}_{3,N}} (f,g)(x), 
%\end{eqnarray}
%where $\Omega^N=\{e^{\frac{\pi ik}{N}}:\;k=0,\dots,N\}$ is the set of uniformly distributed directions. Indeed, this assertion follows from an easy observation that for any rectangle $R\in \mathcal{R}_{1,N}$, there exists a rectangle $\tilde{R}\supset R$ such that the longest side of $\tilde{R}$ is parallel to $e^{\frac{\pi ik}{N}}$ for some $k\in\{0,\dots,N\},$ see ~Grafakos~\cite[Exercise 5.2.3]{Gra2}. 

Another important class of maximal functions in the bilinear theory,  as studied by Lacey~\cite{Lac}, is defined as 
\begin{equation}\label{Lacey}
	\mathcal{ M}_{\alpha}(f,g)(x)=\sup\limits_{t>0}\frac{1}{2t}\int_{-t}^{t}|f(x-t)g(x-\alpha t)|\;dt.
\end{equation}
where $\alpha\in\R$. Observe that if $\alpha=0,1$, then $L^p-$estimates for $\mathcal{ M}_{\alpha}$ can be easily deduced from that of the Hardy-Littlewood maximal function. In the remaining cases, Lacey
~ \cite{Lac} proved that for $\alpha\neq 0,1$, the operator $\mathcal{M}_{\alpha}$ maps $L^{p_1}(\R)\times L^{p_2}(\R)$ into $L^{p_3}(\R)$ for $\frac{2}{3}<p_3<\infty$ and $\frac{1}{p_3}=\frac{1}{p_1}+\frac{1}{p_2}$. It is well-known that the maximal function $\mathcal{ M}_{\alpha}$ is intimately connected with the bilinear Hilbert transforms. 

Next, we observe that with the use of Lebesgue differentiation theorem, one can deduce that 
\begin{eqnarray*}
	\mathcal{\M}_{\alpha}(f,g)(x)\lesssim\mathcal{M}_{\mathcal R^{\Omega_{\alpha}}}(f,g)(x) \lesssim \mathcal{ M}_{\alpha}(f,Mg)(x),
	\end{eqnarray*}
where $\Omega_{\alpha}=\{(1,\alpha)\}.$ 
This relation naturally gives rise to questions and probable methods to address issues related to $L^p-$boundedness of both the maximal functions. 
\begin{remark}We have few observations in order highlighting the dependence on $\alpha$ of the operator $\|\mathcal M_{\alpha}\|_{L^{p_1}\times L^{p_2}\rightarrow L^{p_3}}$ when $(1,\alpha)$ is near the diagonal. It is evident from the example given in \Cref{example} that the operator $\mathcal{M}_{\mathcal R^{\Omega_1}}$ fails to be bounded from $L^{p}(\R)\times L^{p'}(\R)$ into $L^1(\R)$. A modification in example as given below, allows us to show that the $\|\mathcal{M}_{\mathcal R^{\Omega_{\alpha_N}}}\|_{L^{p}\times L^{p'}\to L^1},$ where $\alpha_N=1-\frac{c}{N},$ grows logarithmically in $N$. This shows that $\|\mathcal{M}_{\alpha}\|_{L^{p}\times L^{p'}\to L^1}$ is not uniformly bounded in the neighbourhood of $\alpha=1$.
\end{remark}
\noindent
{\bf Observation:} Let $\alpha_N=1-\frac{c}{N},$ where $c$ is a small constant, then  $\|\mathcal{M}_{\mathcal R^{\Omega_{\alpha_N}}}\|_{L^{p}\times L^{p'}\to L^1}\gtrsim \log N.$ 
The following example proves this assertion. 
\begin{example}
	For $N\in \N$ consider $f_N(x)=x^{-\frac{1}{p}}\chi_{_{\{x:3\leq x\leq N\}}}(x)$ and $g_N(x)=x^{-\frac{1}{p'}}\chi_{_{\{x:3\leq x\leq N\}}}(x)$. Note that $\|f_N\|_p\simeq \log N. $
	
	For $2<x<N$, consider the rectangle $R_x$ containing the point $(x,x)$ with dimension $(x-1)\sqrt{1+\alpha_N^2}\times \frac{(x-1)\sqrt{1+\alpha_N^2}}{N}$ and having its longest side in the direction of $z=\alpha_Ny$ with $(x,x)$ and $(1,1+\frac{c}{N}(x-1))$ as two vertices on the longer side. Note that the long sides of $R_x$ are given by  $z=\alpha_N^2y+\frac{c}{N}x-\frac{(x-1)\sqrt{1+\alpha_N^2}}{N}$ and $z=\alpha_N y+\frac{c}{N}x.$ Consider
	\begin{eqnarray*}
		\mathcal{M}_{\mathcal R^{\Omega_{\alpha_N}}}(f_N,g_N)(x)&\geq& \frac{1}{|R_x|}\int_{1+\frac{N-c}{\sqrt{(N-c)^2+N^2}}}^{x}\int_{\alpha_N y+\frac{c}{N}x-\frac{(x-1)\left(1+\alpha_N^2\right)}{N}}^{\alpha_N y+\frac{c}{N}x} f_N(y)g_N(z) dzdy\\
		&\geq& \frac{1}{|R_x|}\int_{1+\frac{N-c}{\sqrt{(N-c)^2+N^2}}}^{x}\int_{\alpha_Ny+\frac{c}{N}x-\frac{(x-1)\left(1+\alpha_N^2\right)}{N}}^{\alpha_N y+\frac{c}{N}x} \frac{1}{y^{\frac{1}{p}}(\alpha_Ny+\frac{c}{N}x)^\frac{1}{q}} dzdy\\
		&\geq& \frac{1}{|R_x|}\int_{1+\frac{N-c}{\sqrt{(N-c)^2+N^2}}}^{x}\frac{(x-1)\left(1+\alpha_N^2\right)}{N} \frac{1}{\alpha_N y+\frac{c}{N}x} dy\\
		&=&\frac{1}{\alpha_N(x-1)}\int_{\alpha_N\left(1+\frac{N-c}{\sqrt{(N-c)^2+N^2}}\right)}^{\alpha_Nx}\frac{1}{s+\frac{c}{N}x} ds\\
		&=&\frac{1}{\alpha_N(x-1)}\left[\log x-\log\left(\alpha_N\left(1+\frac{N-c}{\sqrt{(N-c)^2+N^2}}\right)+\frac{c}{N}x\right)\right]
	\end{eqnarray*}
Since $N$ is large and $c$ is a fixed small constant, then for $2<x<N$ we see that  $$\alpha_N\left(1+\frac{N-c}{\sqrt{(N-c)^2+N^2}}\right)+\frac{c}{N}x\simeq 1.$$ 
Thus, we get that
	$$\mathcal{M}_{\mathcal R^{\Omega_{\alpha_N}}}(f_N,g_N)(x)\geq\frac{\log x}{x-1}\simeq \frac{\log\ x}{x}.$$
This implies that 
	$$\|\mathcal{M}_{\mathcal R^{\Omega_{\alpha_N}}}(f_N,g_N)\|_1\gtrsim \int_2^{N}\frac{\log x}{x}=(\log N)^2-(\log2)^2.$$
	\qed
	
%The following questions arise naturally.\\
%
%\noindent 
%{\bf Question 1:} What is the range of $\alpha$ and exponents $p_1, p_2,$ and $p_3$ for which uniform bounds for $\|\mathcal{M}_{\alpha}\|_{L^{p_1}\times L^{p_2}\to L^{p_3}}$ holds?\\
%\noindent 
%{\bf Question 2:} Let $\Omega^N=\{(1,\alpha_i)\}_{i=1}^N$. How does the norm  $\|\sup\limits_{i=1}^N\mathcal{M}_{\alpha_i}\|_{L^{p_1}\times L^{p_2}\to L^{p_3}}$ depend on $N$?
\section*{Acknowledgement}  Ankit Bhojak and Saurabh Shrivastava acknowledge the financial support from Science and Engineering Research Board, Department of Science and Technology, Govt. of India, under the scheme Core Research Grant, file no. CRG/2021/000230. Surjeet Singh Choudhary is supported by CSIR(NET), file no.09/1020(0182)/2019- EMR-I for his Ph.D. fellowship. 

\end{example}

\end{document}